\documentclass[conference, twocolumn, 10pt]{IEEEtran}
\usepackage{amsfonts,amsmath,amssymb,amsthm}
\IEEEoverridecommandlockouts
\usepackage{times}
\usepackage{bm}
\usepackage{graphicx}
\usepackage{setspace}
\usepackage{comment}
\usepackage{cases}
\graphicspath{{./fig_eps/}}
\usepackage{here}
\usepackage{ascmac} 
\usepackage{array}
\usepackage{float}
\usepackage{balance}
\usepackage[linesnumbered, ruled]{algorithm2e}
\usepackage{cite}
\usepackage{url}
\usepackage{subfigure}
\newtheorem{mydef}{Definition}
\newtheorem{mylem}{Lemma}
\newtheorem{mythm}{Theorem}
\newtheorem{mypro}{Problem}
\newtheorem{myprop}{Proposition}

\newtheorem{myrem}{Remark}

\newcommand{\rfig}[1]{Fig.\,\ref{#1}} 
 
\newcommand{\req}[1]{\eqref{#1}}

\newcommand{\rlem}[1]{Lemma\,\ref{#1}}

\newcommand{\rsec}[1]{Section\,\ref{#1}}
\newcommand{\rpro}[1]{Problem\,\ref{#1}}
\newcommand{\rprop}[1]{Proposition\,\ref{#1}}
\newcommand{\rdef}[1]{Definition\,\ref{#1}}

\newcommand{\rthm}[1]{Theorem\,\ref{#1}}

\newcommand{\qedwhite}{\hfill \ensuremath{\Box}}

\begin{document}
\title{\LARGE \textbf{Energy-aware networked control systems \\ under temporal logic specifications}}
\author{Kazumune Hashimoto, Shuichi Adachi, and Dimos V. Dimarogonas
\thanks{Kazumune Hashimoto is with Department of Applied Physics and Physico-Informatics, Keio University, Yokohama, Japan. His work is supported by Grant-in-Aid for JSPS Research Fellow (Grant Number: 17J05743).}
\thanks{Shuichi Adachi are with Department of Applied Physics and Physico-Informatics, Keio University, Yokohama, Japan.}
\thanks{Dimos V. Dimarogonas is with the ACCESS Linnaeus Center,
School of Electrical Engineering, KTH Royal Institute of Technology,  Stockholm, Sweden. His work was supported by the Swedish Research Council (VR), Knut och Alice Wallenberg foundation (KAW), and the H2020 ERC Starting Grant BUCOPHSYS. 
}}
\maketitle

\begin{abstract}
In recent years, event and self-triggered control have been proposed as energy-aware control strategies to expand the life-time of battery powered devices in Networked Control Systems (NCSs). 
In contrast to the previous works in which their control objective is to achieve stability, this paper presents a novel energy-aware control scheme for achieving \textit{high level specifications}, or more specifically, \textit{temporal logic specifications}. Inspired by the standard hierarchical strategy that has been proposed in the field of formal control synthesis paradigm, we propose a new abstraction procedure for jointly synthesizing control and communication strategies, such that the communication reduction in NCSs and the satisfaction of the temporal logic specifications are guaranteed. The benefits of the proposal are illustrated through a numerical example. 

\end{abstract}

\section{Introduction}
The increased popularity of introducing Networked Control Systems (NCSs) where sensors, actuators, and controllers are spatially distributed over communication network, have recently suggested the use of \textit{event-triggered} and \textit{self-triggered control} \cite{heemels2012a,tabuada2010a}. 
In the control strategies, sensor measurements are transmitted to the controller {only when} they are needed based on the violation of prescribed control performances. Applying these strategies is useful, since they have the potential to save energy expenditures by alleviating the communication load for NCSs. Some experimental validations have also been provided, e.g., in \cite{araujo2013}. 

In this paper, we present a new aperiodic formulation for NCSs, which allows them to achieve not only the communication reduction, but also the satisfaction of \textit{high level (temporal) specifications}. To illustrate the proposal more in detail, note first that the control objective in the previous aperiodic formulations is only \textit{stabilization} or \textit{tracking to a given reference}. However, considering the fact that the introduction of NCSs allows us to construct more complex, advanced control architectures as exemplified by automonous robots, traffic systems and so on, achieving the above control objectives may not be sufficient. 
For instance, consider \textit{cloud robotics}\cite{ben2015a}, where autonomous robots such as UAVs and warehouse robots interact with the remote operators for supporting various motion tasks and missions. Although their motion tasks given by operators may be to achieve only stability (e.g., ``go to the region $A$''), it is sometimes required to achieve more complex tasks involving \textit{temporal} ones, such as sequential tasks (e.g., ``visit the region $A$, \textit{and then} $B$, \textit{and then} $C$''), {recurrence} tasks with obstacle avoidance (e.g.,``survey the region $A, B, C, D$ (in any order) \textit{infinitely often}, while \textit{always} avoiding the region $E$''), and other complex tasks with temporal properties (e.g., ``visit $A, B, C, D$, and make sure to avoid $D$ \textit{unless} $A$ is visited'').  
Therefore, it may be of great importance to design the event (self)-triggred strategies for accomodating such high level specifications. 

The goal of this paper is to jointly synthesize a \textit{communication strategy} such that the communication reduction is achieved for NCSs, and a \textit{control strategy} such that the above high level missions can be achieved. In particular, we consider that the specifications are expressed by Linear Temporal Logic (LTL) formula\cite{baier}, and we employ the {hierarchical} approach that has been proposed in field of formal control synthesis paradigm, see, e.g.,  \cite{belta2008a,gerogios2009a,temporalcontrol2009a}. 
First, we construct a {finite transition system} that represents an \textit{abstracted} behavior of the original control system. The transition system is constructed through \textit{reachability analysis}, which evaluates if the state can be steered between each pair of the regions of interest. One of the main contributions of this paper is to provide a new reachability analysis such that the communication strategies can be designed after this procedure. The proposed approach mainly consists of the two steps; \textit{trajectory generation} and \textit{tube generation}. In the trajectory generation, a nominal state trajectory is generated by implementing existing strategies, such as sampling-based algorithms \cite{lavalle1999}. The trajectory is utilized as a {reference} that the actual trajectory should follow to achieve the entrance to the target region. 
In the tube generation, we seek to generate a sequence of tubes (polytopes) among the regions of interest. Since the generated tubes represent safety margins to guarantee reachability, we will make use of them as criteria to analyze reachability. Based on the reachability analysis, we next propose to derive a self-triggered strategy that the controller iteratively assigns communication times by evaluating the self-triggered condition. The self-triggered condition is derived by making use of the tubes obtained by the reachability analysis, and the communication is triggered only when the
state trajectory does not guarantee the satisfaction of the
formula.

The rest of the paper is organized as follows. We describe some preliminaries and the problem formulation in Section~II and III, respectively. 
In Section~IV, reachability analysis and an algorithm to obtain a finite transition system are given. In Section~V, we propose the implementation algorithm involving both high and low level strategies. In Section~VI, a simulation example is given to validate the effectiveness of the proposed approach. We finally conclude in Section~VII. 

\smallskip
\noindent
\textbf{Notations.} Let $\mathbb{R}_+$, $\mathbb{N}$, $\mathbb{N}_+$, $\mathbb{N}_{a: b}$ be the {positive real, non-negative}, {positive integers}, and the set of integers in the interval $[a, b]$, respectively. 
For vectors $v_1, \ldots, v_N$, denote by ${\rm co} \{ v_1, \ldots, v_N \}$ their \textit{convex hull}. 
For two given sets ${A}\subset {\mathbb{R}}^{n}$, $B\subset {\mathbb{R}}^{n}$, denote by $A \oplus B$ the \textit{Minkowski sum} $A \oplus B = \{ z \in {\mathbb{R}}^{n} \ |\ \exists x\in A, y\in B : z = x + y \}$ and by $A \ominus B$ the \textit{Pontryagin difference} $A \ominus B = \{ x\in {\mathbb{R}}^{n}\ |\ x+y \in A, \ \forall y\in B \}$.  
For simplicity we denote the collection of $N$ sets ${\cal X}_1,\ldots,{\cal X}_N \subset {\mathbb{R}}^{n}$ as ${\cal X}_{1:N} = \{ {\cal X}_1,\ldots,{\cal X}_N \}$. 


\section{Preliminaries} \label{preliminaries}
In this section, we provide some preliminaries on \textit{finite transition systems} and \textit{Linear Temporal Logic (LTL)} formulas. 
A {finite transition system} is a tuple ${\cal T} = ({S}, s_{init}, \delta, \Pi, g)$, where $S$ is a set of states, $s_{init} \in S$ is an initial state, $\delta \subseteq S\times S$ is a transition relation, $\Pi$ is a set of atomic propositions, and $g: S \rightarrow 2^\Pi$ is a labeling function. A path of ${\cal T}$ is an infinite sequence of states $s_{seq} = s_0 s_1 s_2 \cdots$, such that $s_0 = s_{init}$, $(s_i, s_{i+1}) \in \delta$, $\forall i \in \mathbb{N}$. A trace of a path $s_{seq} = s_0 s_1 s_2 \cdots$ is given by ${\rm trace} (s_{seq}) = g(s_0) g(s_1) g(s_2)\cdots$. Linear Temporal Logic (LTL) is defined over a set of atomic propositions $\Pi$ and is useful to express various task specifications involving temporal properties. The LTL formulas is given according to the following grammer: 
$\phi::= {\rm true}\ |\ \pi\ |\ \phi_1 \wedge \phi_2 \ |\ \neg \phi \ |\ \phi_1 \mathsf{U} \phi_2$, 
where $\pi \in \Pi$ is the atomic proposition, $\wedge$ (and), $\neg$ (negation) are the boolean operators, and $\mathsf{U}$ (until) is the temporal operator. We can also derive other useful temporal operators, such as $\Box$ (always), $\Diamond$ (eventually), and $\implies$ (implication). For the LTL semantics, see Chapter~5 in \cite{baier} for details. For a given path $s_{seq}$ of ${\cal T}$, we denote ${\rm trace}(s_{seq}) \models \phi$ if and only if ${\rm trace}(s_{seq})$ satisfies $\phi$. The path $s_{seq}$ is called accepting if and only if ${\rm trace}(s_{seq}) \models \phi$. 

\section{Problem formulation}
\subsection{Plant dynamics}
We consider a Networked Control System illustrated in \rfig{NCS}, where the plant and the controller are connected over communication channels. The controller system is responsible for both implementing a \textit{high level planner} that generates a symbolic plan for a given LTL specification $\phi$, and a \textit{low level planner} that generates feedback controllers to operate the plant. How these planners are designed will be formally
given later in this paper. 
The dynamics of the plant are given by the following Linear-Time-Invariant (LTI) systems: 
\begin{align}
x_{k+1} = A x_k + B u_k + w_k, \label{sys}
\end{align}
for $k\in \mathbb{N}$, where $x_k \in \mathbb{R}^n$ is the state, $u_k \in \mathbb{R}^m$ is the control variable, and $w_k \in \mathbb{R}^n$ is the additive disturbance. We assume that the control and the disturbance variables are constrained in the sets ${\cal U}$ and ${\cal W}$, i.e., 
\begin{equation}\label{uwsets}
u_k \in {\cal U}, \ \ w_k \in {\cal W}
\end{equation}
for all $k \in \mathbb{N}$. Here, ${\cal U}$ and ${\cal W}$ are both polytopic sets containing the origin in their interiors. 
Regarding the state, we assume $x_k \in {\cal X}$, $\forall k \in \mathbb{N}$, where ${\cal X}$ is a {polygonal} set that can be either a convex or non-convex region. The set ${\cal X}$ represents the \textit{free space}, in which the state is allowed to move. 
Inside ${\cal X}$, we assume that there exist in total $N_I$ number of polytopic regions ${\cal R}_1, {\cal R}_2, \ldots, {\cal R}_{N_I} \subset {\cal X}$, which represent the \textit{regions of interest} in ${\cal X}$. These regions are assumed to be disjoint, i.e., ${\cal R}_i \cap {\cal R}_j = \emptyset$, $\forall i, j \in \mathbb{N}_{1:N_I}$ with $i\neq j$. Moreover, let $x_{cent,i} \in {\cal R}_i$, $i \in \mathbb{N}_{1:N_I}$ denote the Chebyshev center \cite{borrelli} of the polytope ${\cal R}_i$. The Chebyshev center is the center of the maximum ball that is included in ${\cal R}_i$ and is obtained by solving a linear program (for details, see Section 5.4.5 in \cite{borrelli}). For simplicity, denote by ${\cal R} = \{{\cal R}_1, {\cal R}_2, \ldots, {\cal R}_{N_I}\}$ the collection of all regions of interest. In addition, the initial state is assumed to be inside one of the regions of interest, i.e., $x_{0} \in {\cal R}_{init}$ with ${\cal R}_{init} \in {\cal R}$. 

Let $\pi_i$, $i \in \mathbb{N}_{1:N_I}$ be the atomic proposition assigned to the region ${\cal R}_i$. Namely, $\pi_i$ holds true if and only if $x \in {\cal R}_i$. Also, denote by $\pi_0$ an atomic proposition associated to the regions of \textit{non-interest}, i.e., $\pi_0$ holds true if and only if $x \in {\cal X} \backslash (\cup^{N_I} _{i=1} {\cal R}_i) $. 
The atomic proposition $\pi_0$ represents a \textit{dummy symbol}, which will not be used to describe the task specification. Let $\Pi = \{\pi_1, \pi_2, \ldots, \pi_{N_I}\}$ and ${h}_X : \mathbb{R}^n \rightarrow 2^\Pi $ be the mapping that maps the state information to the corresponding atomic proposition, i.e., 
\begin{numcases}
{{h}_X (x) = }
	\pi_i, \ \ {\rm if}\ \ x \in {\cal R}_i, \ \ i\in \mathbb{N} _{1:N_I}, \\
	\pi_0, \ \ {\rm if}\ \ x \in {\cal X}_{\backslash {\cal R}}, 
\end{numcases}
where ${\cal X}_{\backslash {\cal R}} = {\cal X}\backslash(\cup^{N_I} _{i=1} {\cal R}_i)$. 

\begin{figure}[t]
  \begin{center}
   \includegraphics[width=7.5cm]{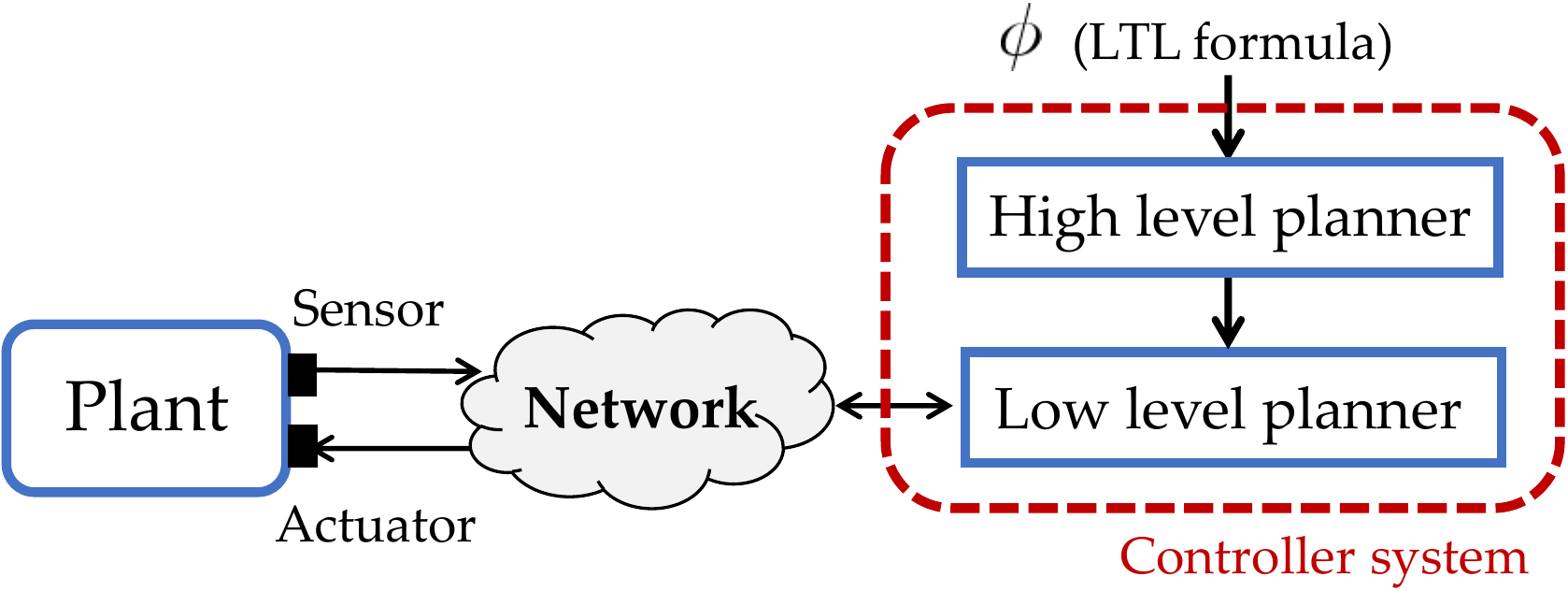}
   \caption{Networked Control System.} 
   \label{NCS}
  \end{center}
 \end{figure}

\subsection{Satisfaction relation over the state trajectory}
Denote by ${\bf {x}} = x_0 x_1 x_2 \cdots $ a state trajectory in accordance to \req{sys}, with $x_k \in {\cal X}$,  $u_k \in {\cal U}$, $w_k \in {\cal W}$, $\forall k \in \mathbb{N}$. We next define the satisfaction of the formula $\phi$ by the state trajectory ${\bf {x}}$. To this end, we define the \textit{trajectory of interest} as follows: 
\begin{mydef}\label{trajectry_of_interest}
\normalfont
Given a state trajectory ${\bf {x}} = x_0 x_1 x_2 \cdots $, the \textit{trajectory of interest} ${\bf {x}}_I$ corresponding to ${\bf {x}}$ is defined as ${\bf {x}}_I = x_{k_0} x_{k_1} x_{k_2} \cdots $ with $k_j < k_{j+1}$, $\forall j \in \mathbb{N}$, such that: $h_X (x_{k_j}) \in \Pi$, $\forall j \in \mathbb{N}$, and $h_X (x_k) = \pi_0$, $\forall k \in (k_j, k_{j+1})$, $\forall j \in \mathbb{N}$. \qedwhite
\end{mydef}

Stated in words, the trajectory of interest is given by eliminating all states that belong to the regions of non-interest. The trace of the state trajectory is generated based on the trajectory of interest as defined next: 


\begin{mydef}\label{trace_traj}
\normalfont
Given a state trajectory ${\bf {x}} = x_0 x_1 \cdots $, the \textit{trace} of ${\bf {x}}$ is given by ${\rm trace} ({\bf {x}})= \rho_0 \rho_1 \cdots $, which is generated over the corresponding trajectory of interest ${\bf {x}}_I = x_{k_0} x_{k_1} x_{k_2} \cdots $, satisfying the following rules for all $\ell \in \mathbb{N}$, $i \in \mathbb{N}$: 
\begin{enumerate}
\item $\rho_0 = h_X(x_{k_0})$; 
\item If $\rho (\ell) = h_X (x_{k_i})$ 
and there exists $j > i$ such that $h_X (x_{k_i}) = h_X (x_{k_{i+1}}) = \cdots = h_X (x_{k_{j}})$, $h_X (x_{k_{j}}) \neq h_X (x_{k_{j+1}})$, then $\rho_{\ell +1} = h_X (x_{k_{j+1}})$. 
\item If $\rho_\ell = h_X (x_{k_i})$, 
and $h_X (x_{k_j}) = h_X (x_{k_i})$, $\forall j \geq i$, then $\rho_m = \rho_\ell$, $\forall m \geq \ell$. \qedwhite
\end{enumerate}
\end{mydef}
For example, assume that $x_k \in {\cal R}_1$ for $k = 0, 1, 2$, $x_k \in {\cal X}_{\backslash {\cal R}}$ for $k = 4, 5$ and $x_k \in {\cal R}_2$ for $k = 6, 7, 8 \cdots$. This means that the state initially starts from ${\cal R}_1$, leave ${\cal R}_1$ for entering ${\cal R}_2$, and remains there for all the time afterwards. 
The trajectory of interest is given by $x_0 x_1 x_2 x_6 x_7 x_8 \cdots $. The trace of the trajectory according to \rdef{trace_traj} is $\rho = \rho_0 \rho_1 \rho_2 \cdots $ with $\rho_0 = h_X(x_0) = \pi_1$ and $\rho_\ell = \pi_2$, $\forall \ell \geq 1$. 

\begin{mydef}\label{satisfaction_formula}
Given a state trajectory ${\bf x} = x_0 x_1 x_2 \cdots $, we say that ${\bf x}$ satisfies the formula $\phi$ if and only if the corresponding trace according to \rdef{trace_traj} satisfies $\phi$, i.e., $\pi_{seq} = {\rm trace} ({\bf x}) \models \phi$. 
\end{mydef}

\subsection{Communication strategy}
During online implementation, the plant needs to receive suitable control signals from the networked controller for achieving the LTL specification $\phi$. 
In this paper, we employ a \textit{self-triggered strategy}\cite{heemels2012a}, which is known to be one of the most useful energy-aware communication protocols. 
To provide the overview, suppose that at $k \in \mathbb{N}$ the plant transmits the state information $x_k$ to the controller. 
Then, the controller not only computes suitable control signals to be applied, but also the next communication time $k' > k$. Once $k'$ is determined, a set of control actions from $k$ to $k'-1$, say $u_k, \ldots, u_{k'-1}$, are transmitted in packets to the plant, and these are applied in an open-loop fashion. Thus, no communication occurs until the next communication time $k'$. Moreover, as the controllers are given in an open-loop fashion, the plant does not need to measure the state information until $k'$. As a consequence, both sensor and communication systems do not need to be used for all times until the next communication time, and, therefore, energy savings of battery powered devices can be achieved. 

\subsection{Problem formulation}
The goal of this paper is to design control and self-triggered strategies, such that the resulting state trajectory satisfies the desired LTL formula. That is: 
\begin{mypro}\label{problem_formulation}
\normalfont 
For a given LTL specification $\phi$, design both control and self-triggered strategies, such that the resulting state trajectory satisfies $\phi$. \qedwhite 
\end{mypro}

\section{Constructing transition system} 
As a first step to solve \rpro{problem_formulation}, we construct a finite transition system that represents an \textit{abstracted} model to describe the behavior of the control systems in \req{sys}. Specifically, we aim at obtaining ${\cal T} =  (S , s_{init} , \delta, \Pi, g )$, where $S = \{ s_1, \ldots, s_{N_I} \}$ is a set of symbolic states, $s_{init}  \in S  $ is an initial state, $\delta \subseteq S \times S  $ is a transition relation, $\Pi = \{\pi_1, \ldots, \pi_{N_I}\}$ is a set of atomic propositions, and $g  : S  \rightarrow 2^\Pi$ is a labeling function. Roughly speaking, each symbol $s_i \in S$ indicates the region of interest ${\cal R}_i$ (i.e., the region having the same index $i$). To relate each symbol to the corresponding region of interest, let $\Gamma : S  \rightarrow {\cal R}$ be the mapping given by 
\begin{equation}\label{gamma}
\Gamma ({s}_i) = {\cal R}_i, \ \ \forall i \in \mathbb{N}_{1:N_I}. 
\end{equation}
Conversely, let $\Gamma^{-1}: {\cal R}  \rightarrow S $ be the mapping from each region of interest to the corresponding symbolic state. The symbol $s_{init} \in S$ represents the symbolic state associated with ${\cal R}_{init}$, i.e., $s_{init} = \Gamma^{-1} ({\cal R}_{init})$. The labeling function $g (s_i)$ outputs the atomic proposition assigned to ${\cal R}_i$, i.e., $g(s_i ) = \pi_i$. The transition relation $(s_i, s_j) \in \delta$ indicates that every $x\in {\cal R}_i$ can be steered to ${\cal R}_j$ in finite time. A more formal definition of $\delta$ is provided below.

\subsection{Definition of reachability} 
To characterize $\delta$ in the transition system, let us analyze the reachability among the regions of interest. To this end, consider a pair $({\cal R}_i, {\cal R}_j) \in {\cal R} \times  {\cal R}$ with $i\neq j$. For notational simplicity, let ${\cal X}_{ij} \subset {\cal X}$ be given by 
\begin{equation}\label{xij}
{\cal X}_{ij} = {\cal X} \backslash \bigcup_{n\in \mathbb{N}_{\backslash ij}} {\cal R}_{n},  
\end{equation}
where $\mathbb{N}_{\backslash ij} = \{1, \ldots, N_I \}\backslash \{i, j \}$. 
That is, ${\cal X}_{ij}$ represents the free space that we exclude all regions of interest other than ${\cal R}_i$ and ${\cal R}_j$. Note that ${\cal X}_{ij}$ is a polygonal set that can be a non-convex region. 
Whether the transition is allowed in ${\cal T}$, from $s_i = \Gamma^{-1} ({\cal R}_i)$ to $s_j = \Gamma^{-1} ({\cal R}_j)$ (i.e., $(s_i, s_j) \in \delta$), is determined according to the following notion of \textit{reachability}: 

\begin{mydef}[Reachability]\label{reachable_def}
\normalfont 
The state is \textit{reachable} from ${\cal R}_i$ to ${\cal R}_j$ ($i\neq j$), which we denote by $(s_i, s_j) \in \delta$, if there exists a finite $L \in \mathbb{N}_+$ such that the following holds: for any $x_0 \in {\cal R}_i$ and the disturbance sequence $w_0, w_1, \ldots, w_{L-1} \in {\cal W}$, there exist $u_0, u_1, \ldots, u_{L-1}\in {\cal U}$ such that the resulting state trajectory $x_1, \ldots, x_L$ in accordance with \req{sys} satisfies 
\begin{enumerate}
\renewcommand{\labelenumi}{(C.\arabic{enumi})}
\item $x_L\in  {\cal R}_j$,
\item $x_\ell \in  {\cal X}_{ij}, \ \forall \ell \in \mathbb{N}_{0:L}$, 
\item If $x_{\ell'} \in {\cal R}_j$ for some $\ell' \in \mathbb{N}_{1:L}$, then $x_{\ell} \notin {\cal R}_i$, $\forall \ell \in \mathbb{N}_{\ell':L}$. \qedwhite
\end{enumerate}
\end{mydef}
Based on \rdef{reachable_def}, reachability holds from ${\cal R}_{i}$ to ${\cal R}_j$ if there exists a controller such that any state in ${\cal R}_i$ can be steered to ${\cal R}_j$ in finite time. Moreover, we require by (C.2) that the state needs to {avoid} any other region of interest apart from ${\cal R}_i$ and ${\cal R}_j$. Also, (C.3) implies that once the state enters ${\cal R}_j$ it must not enter ${\cal R}_i$ afterwards. The conditions (C.2), (C.3) are essentially required to guarantee that the \textit{trace} of the state trajectory satisfies the following property: 
\begin{myprop}\label{trace_prop}
For every $x_0 \in {\cal R}_i$, the trace of the state trajectory $x_0, x_1, \ldots, x_L$ satisfying (C.1) -- (C.3) is $\pi_i \pi_j$. 
\end{myprop}
\rprop{trace_prop} implies that the {trace} of the state trajectory satisfying (C.1) -- (C.3), which is generated according to the rules in \rdef{trace_traj}, is consistent with the trace for the trasition from $s_i$ to $s_j$ (i.e., $g(s_i) g(s_j)$). This consistency is important, since otherwise the accepting trace over ${\cal T}$ might not lead to the acceptance of the trace over the actual state trajectory. For example, the state trajectory $x_0x_1x_2$ $(L=2)$ satisfying $x_0 \in {\cal R}_i$, $x_{1} \in {\cal R}_m$, $x_{2} \in {\cal R}_j$ with $m \neq i, j$, which does \textit{not} satisfy the condition (C.2), yields the corresponding trace according to \rdef{trace_traj} as $\pi_i \pi_m \pi_j (\neq \pi_i \pi_j)$. In this case, the formula such as $\phi = \Box \neg \pi_m$ may be satisfied by $\pi_i \pi_j$ (the trace of ${\cal T}$), while, on the other hand, not be satisfied by $\pi_i \pi_m \pi_j$ (the trace of the trajectory). As another example, consider $x_0x_1x_2x_3$ $(L=3)$ satisfying $x_0 \in {\cal R}_i$, $x_{1} \in {\cal R}_j$, $x_{2} \in {\cal R}_i$ $x_3 \in {\cal R}_j$ (i.e., the trajectory traverses ${\cal R}_i{\cal R}_j{\cal R}_i{\cal R}_j$). In this case, the condition (C.3) is violated since the state enters ${\cal R}_i$ after reaching ${\cal R}_j$. The corresponding trace here is $\pi_i \pi_j \pi_i \pi_j (\neq \pi_i \pi_j)$. Since the accepting path of ${\cal T}$ satisfying $\phi$ is designed by the high level planner (see \rsec{implementation}), and the resulting state trajectory is generated based on the accepting path, the above inconsistencies could lead to the violation of $\phi$ by the state trajectory. 

Note also that our definition of reachability differs from the ones presented in \cite{belta2008a}, which, in contrast to (C.2), requires that once the state enters ${\cal R}_j$ it must remain there afterwards (i.e., if $x_{\ell'} \in {\cal R}_j$ for some $\ell' \in \mathbb{N}_{1:L}$ then $x_{\ell} \in {\cal R}_j$, $\forall \ell \in \mathbb{N}_{\ell':L}$). 
Indeed, such condition is more restrictive than (C.2) and is not required in this paper since we allow the states to enter the regions of non-interest while moving from ${\cal R}_i$ to ${\cal R}_j$. 
For example, consider $x_0x_1x_2 x_3 x_4$ satisfying $x_0 \in {\cal R}_i$, $x_{1} \in {\cal X}_{\backslash {\cal R}}$, $x_{2} \in {\cal R}_j$, $x_{3} \in {\cal X}_{\backslash {\cal R}}$, $x_{4} \in {\cal R}_j$. The corresponding trajectory of interest according to \rdef{trajectry_of_interest} is $x_0x_2 x_4$ and so the resulting trace according to \rdef{trace_traj} is $\pi_i \pi_j$. That is, while the state traverses the regions of non-interest \textit{after} reaching ${\cal R}_j$, the corresponding trace is still given by $\pi_i \pi_j$ (as long as the terminal state belongs to ${\cal R}_j$) since we eliminate all states in ${\cal X}_{\backslash {\cal R}}$ when generating the trace. 

\subsection{Proposed algorithm to analyze reachability}\label{reachability_sec}
This section provides a way to analyze reachability from ${\cal R}_i$ to ${\cal R}_j$. The proposed algorithm consists of the two steps; \textit{trajectory generation} and \textit{tube generation}. The following subsections are provided to described these steps in detail. 
\subsubsection{Trajectory generation}\label{trajectory_generation_section}
In the trajectory generation step, we generate a sample state trajectory from the \textit{nominal} system of \req{sys}, starting from a given $\hat{x}_0 \in {\cal R}_{i}$ to the desired target set ${\cal R}_{j}$. Specifically, we produce a state trajectory $\hat{x}_0, \hat{x}_1, \ldots \hat{x}_L$, and the corresponding control $\hat{u}_0, \hat{u}_1, \ldots \hat{u}_{L-1} \in {\cal U}$ for some $L\in {\mathbb{N}}_+$, such that $\hat{x}_0 = x_{cent, i}$ (recall that $x_{cent, i}$ represents the Chebyshev center of ${\cal R}_i$), 
\begin{equation}
\hat{x}_{k+1} = A \hat{x}_k + B \hat{u}_k \in {\cal X}_{ij},\ \ \forall k \in {\mathbb{N}}_{1:L-1}, 
\end{equation}
and $\hat{x}_L \in {\cal R}_{j}$. The illustration of the trajectory is shown in \rfig{gen_traj_step}. Roughly speaking, the trajectory provides a \textit{reference} that the actual states should follow to move from ${\cal R}_{i}$ to ${\cal R}_{j}$. Notice that the nominal trajectory must remain in ${\cal X}_{ij}$ which aims to fulfill the condition (C.2) in \rdef{reachable_def}. 

So far, numerous techniques have been proposed to generate the above trajectory. The most popular one may be sampling-based algorithms, such as RRT\cite{lavalle1999}, RRT*\cite{karaman2010} and their variants\cite{dang2008}, which are known to be powerful techniques to the find feasible trajectories regardless of the non-convexity of ${\cal X}_{ij}$. Alternative methods are the optimization-based approach, such as those by solving Mixed Integer Linear Programming (MILP)\cite{collision_free2} or other variants to relax the computational complexity\cite{ono2011}. In view of the many different techniques, how the trajectory generation scheme should be applied is beyond the scope of this paper; we can utilize any techniques to obtain the nominal trajectory for the reachability analysis. 

\begin{figure}[t]
   \centering
    \hspace{-0.52cm}
    \subfigure[Trajectory generation]{
      {\includegraphics[width=4.3cm]{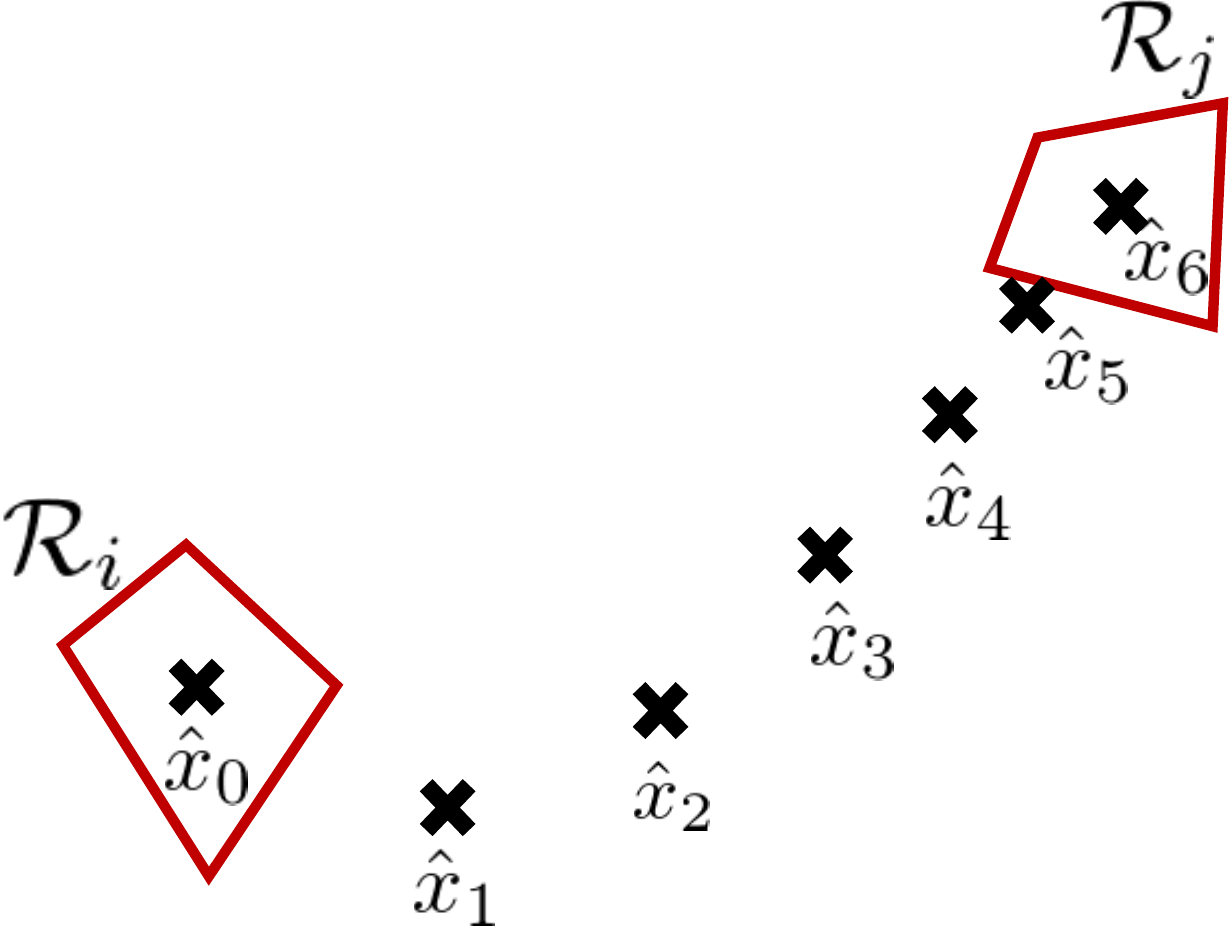}} \label{gen_traj_step}}
    \hspace{-0.5cm}
    \subfigure[Tube generation]{
      {\includegraphics[width=4.3cm]{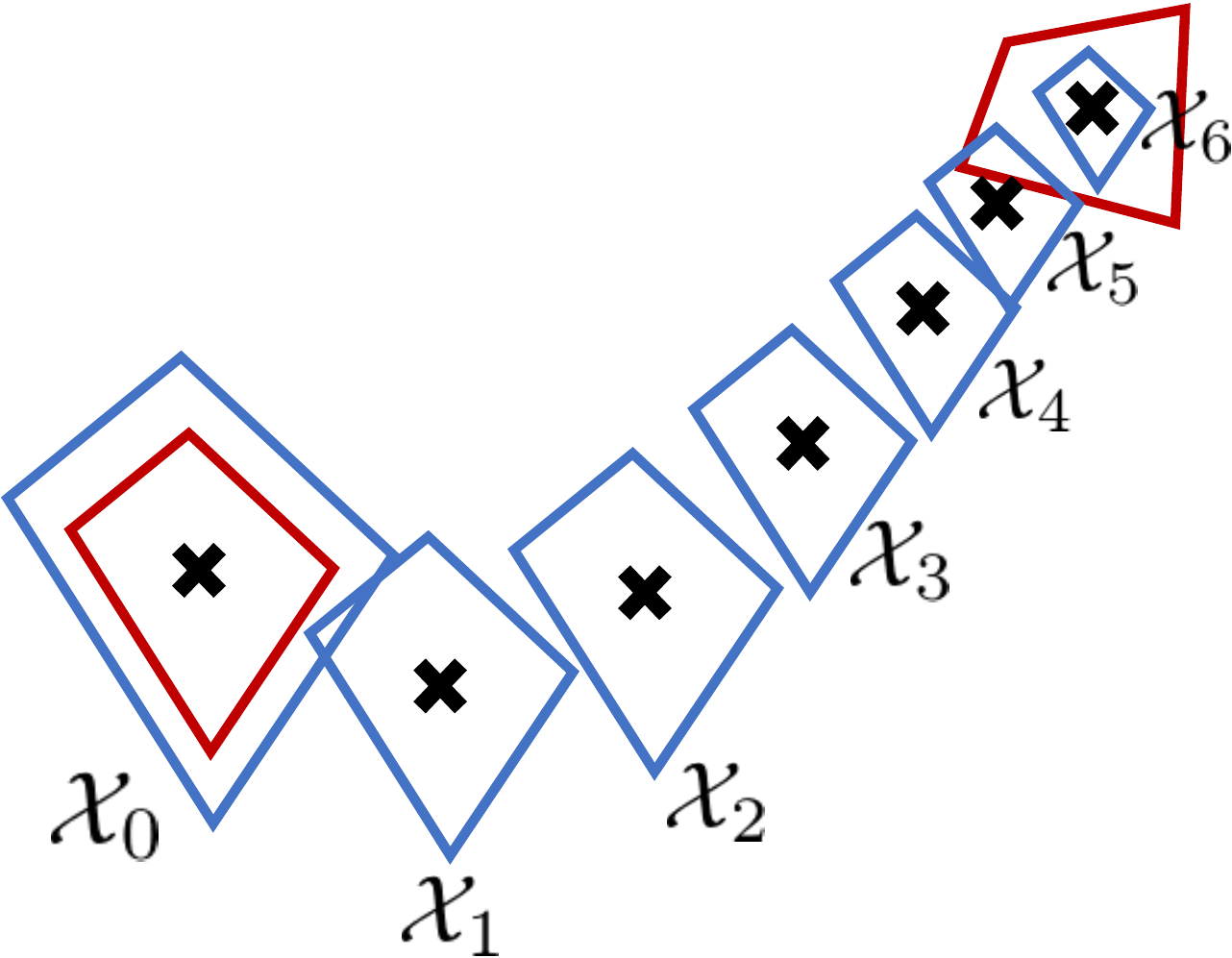}} \label{gen_tube_step}}
      \hspace{-0.5cm}
      \caption{Illustration of the generated trajectory (black marks) and tubes (regions covered by blue lines).}
\end{figure}

\subsubsection{Tube generation}\label{tube_generation}
In the tube generation step, we solve a convex program to generate a sequence of tubes (polytopes) based on the obtained nominal trajectory. The illustration of the generated tubes is shown in \rfig{gen_tube_step}. 
Specifically, we aim to generate a sequence of tubes ${\cal X}_0, {\cal X}_1, \ldots, {\cal X}_{L} \subset {\cal X}$, which are designed as 
\begin{equation}\label{xsets}
{\cal X}_\ell = \hat{x}_\ell \oplus \varepsilon_\ell {\cal Z}_{ij}
\end{equation}
for $\ell \in \mathbb{N}_{0:L}$, where $\varepsilon_\ell >0$ is the scalar decision variable and ${\cal Z}_{ij} = {\rm co} \{ z_1, \ldots, z_{p}\} \subset \mathbb{R}^n$ denotes a given polytope containing the origin in the interior. 
The set ${\cal Z}$ can be arbitrarily chosen if $0 \in {\cal Z}_{ij}$ and is not necessary to be an invariant set. A natural choice may be to characterize as ${\cal Z}_{ij} = -x_{cent,i} \oplus {\cal R}_i$, i.e., select ${\cal Z}_{ij}$ as having the same shape as ${\cal R}_i$, and the Chebyshev center is the origin. From \req{xsets}, each ${\cal X}_\ell$ is a polytope whose center is a nominal state $\hat{x}_\ell$ and the size (volume) is characterized by the scalar $\varepsilon_\ell$. 
Alternatively, each ${\cal X}_\ell$ can be expressed by a convex hull ${\cal X}_\ell = {\rm co} \{ x_{\ell,1}, \ldots, x_{\ell,p}\}$, where 
$x_{\ell,s} = \hat{x}_\ell + \varepsilon_\ell z_s$, $s \in \mathbb{N}_{1:p}$. 
Let $\bar{\varepsilon}_\ell \in \mathbb{R}$, $\ell \in \mathbb{N}_{0:L}$ be given by  
\begin{align}\label{barvarepsilon}
\bar{\varepsilon}_\ell = {\rm max} \{\varepsilon > 0 \ |\ \hat{x}_\ell \oplus \varepsilon {\cal Z}_{ij} \subseteq {\cal X}_{ij} \}. 
\end{align}
That is, $\bar{\varepsilon}_\ell$ denotes the maximum value of $\varepsilon$ such that the corresponding set $\hat{x}_\ell \oplus \varepsilon {\cal Z}_{ij}$ belongs to ${\cal X}_{ij}$. Note that this can be easily computed by searching the maximum value of $\varepsilon$ with the property that the intersection between $\hat{x}_\ell \oplus \varepsilon {\cal Z}_{ij}$ and ${\cal X}_{ij}$ is empty. 

Using the above notations, we propose the following optimization problem: 
\begin{mypro}\label{opt_problem}
For given $\hat{x}_{0:L}$, $\bar{\varepsilon}_{0:L}$ and ${\cal Z}_{ij}$, find ${\cal X}_0, {\cal X}_1, \ldots, {\cal X}_L$ and ${\cal U}_0$, ${\cal U}_1$, $\ldots$, ${\cal U}_L$ with ${\cal U}_\ell = {\rm co} \{u_{\ell, 1}, \ldots, u_{\ell, p}\}$, by solving the following problem : 
\begin{equation}\label{cost_func}
\underset{\varepsilon_{0:L},\ {\cal U}_{0:L}}{\max}\ \varepsilon_0
\end{equation}
subject to 
\begin{numcases}
   { }
     \varepsilon_\ell \leq \bar{\varepsilon}_\ell, \ \  \forall \ell \in \mathbb{N}_{0:L} \label{epsilon_const}\\
     {\cal X}_L \subseteq {\cal R}_{j}, \label{terminal_constraints} \\
     A x_{\ell,s} + B u_{\ell,s} \in {\cal X}_{\ell+1} \ominus {\cal W}, \label{constraint} \\
     u_{\ell ,s} \in {\cal U}, \label{constraint_u}
\end{numcases}
where \req{constraint}, \req{constraint_u} must hold for all $\ell \in \mathbb{N}_{0:L-1} , \ s \in \mathbb{N}_{1:p}$. 
\qedwhite
\end{mypro}

The constraint \req{epsilon_const} indicates that the scalars $\varepsilon_\ell$ must be upper bounded by $\bar{\varepsilon}_\ell$, and is utilized to guarantee that the generated tubes remain inside ${\cal X}_{ij}$. 
The constraint \req{terminal_constraints} indicates that the terminal tube must belong to ${\cal R}_j$, which guarantees that the state at the final time step belongs to ${\cal R}_j$. The constraint \req{constraint} indicates the condition to relate the neighboring tubes; it guarantess that all vertices of the tube ${\cal X}_\ell$ can move to the next one ${\cal X}_\ell$ that is constrained by the disturbance set ${\cal W}$. 
Overall, the problem aims at finding a sequence of tubes satisfying all the constraints as described above, by maximizing the volume of ${\cal X}_0$ (i.e., the scalar $\varepsilon_0$). Here, maximizing the variable $\varepsilon_0$ is motivated by the following main result: 

\begin{mythm}\label{reachable_lem}
Suppose that \rpro{opt_problem} has a solution, providing optimal tubes and controllers denoted as ${\cal X} _{0}, \ldots, {\cal X}_L$, and ${\cal U} _{0}, \ldots, {\cal U}_{L-1}$, respectively. Then, reachability holds from ${\cal R}_i$ to ${\cal R}_j$, if: 
\begin{enumerate}
\renewcommand{\labelenumi}{(D.\arabic{enumi})}
\item ${\cal R}_i \subseteq {\cal X} _0$;
\item if ${\cal X}_{\ell'} \cap {\cal R}_j \neq \emptyset$ for $\ell' \in \mathbb{N}_{1:L-1}$, then ${\cal X}_{\ell} \cap {\cal R}_i = \emptyset$, $\forall \ell \in \mathbb{N}_{\ell':L-1}$. 
\end{enumerate}
\end{mythm}
\begin{proof}
Let ${\cal X} _{0}, \ldots, {\cal X}_L$, ${\cal U} _{0}, \ldots, {\cal U}_{L-1}$ be the optimal solution of the tubes and the control sets by solving \rpro{opt_problem} denoted by ${\cal X} _{\ell} = {\rm co} \{ x _{1, \ell}, x _{2, \ell}, \ldots, x _{p , \ell} \}$, ${\cal U} _{\ell}  = {\rm co} \{ u _{1, \ell}, u _{2, \ell}, \ldots, u _{p , \ell} \}$, $\forall \ell \in \mathbb{N}_{0:L}$, respectively. 
In the following, we first show that for every $x_0 \in {\cal X} _0$, there exist $u_0, u_1, \ldots, u_{L-1} \in {\cal U}$ such that $x_\ell \in {\cal X}_\ell$, $\forall \ell \in \mathbb{N}_{0:L}$. To inductively show this, assume $x_\ell \in {\cal X}_\ell$ for some $\ell \in \mathbb{N}_{0:L}$. Since $x_\ell \in {\cal X}_\ell$,  there exist $\lambda_s \in [0, 1]$, $s\in\mathbb{N}_{1:p}$ such that $x_\ell = \sum^{p} _{s=1} \lambda_s x _{s, p}$ holds. Let $u_{\ell} \in \mathbb{R}^m$ be given by $u_\ell = \sum^{p} _{s=1} \lambda_s u _{s, p} \in {\cal U}$. We then obtain 
\begin{align}
x_{\ell+1} &= A x_\ell + B u_\ell + w_\ell \notag \\
             &= \sum^{p} _{s=1} \lambda_s ( A x _{s,\ell} + B u _{s,\ell} ) + w_\ell \in  {\cal X} _{\ell+1},
\end{align}
$\forall w_\ell \in {\cal W}$, where we have used the fact that $Ax _{s,\ell} + B u _{s, \ell} \in {\cal X} _{\ell+1} \ominus {\cal W}$, $\forall s \in \mathbb{N}_{1:p}$ from \req{constraint}. 
Therefore, we obtain 
\begin{equation}\label{relation_proven}
x_\ell \in {\cal X} _\ell \implies \exists u \in {\cal U},\ \ {\rm s.t.},\ x_{\ell + 1} \in {\cal X} _{\ell+1}, \forall w_\ell \in {\cal W} 
\end{equation}
and the above holds for any $\ell \in \mathbb{N}_{0:L-1}$. Thus, for every $x_0 \in {\cal X} _0$, there exist $u_0, u_1, \ldots, u_{L-1} \in {\cal U}$ such that $x_\ell \in {\cal X}_\ell$, $\forall \ell \in \mathbb{N}_{0:L}$ holds. 

Let us now show that the reachability holds if the conditions (D.1), (D.2) are satisfied. 
Since ${\cal R}_i \subseteq {\cal X}_0$, it follows that for every $x_0 \in {\cal R} _i (\subseteq {\cal X}_0) $, there exist $u_0, u_1, \ldots, u_{L-1} \in {\cal U}$ such that $x_\ell \in {\cal X}_\ell$, $\forall \ell \in \mathbb{N}_{0:L}$ holds.   
Thus, the condition (C.1) in \rdef{reachable_def} holds since ${\cal X}_L \subseteq {\cal R}_j$ from \req{terminal_constraints}. Moreover, from \req{epsilon_const} and \req{barvarepsilon}, it follows that ${\cal X} _\ell \subseteq {\cal X}_{ij}$, $\forall \ell \in \mathbb{N}_{0:L}$, which means that the trajectory with $x_\ell \in {\cal X}_\ell$, $\forall \ell \in \mathbb{N}_{0:L}$ satisfies the condition (C.2). Finally, from (D.2), the trajectory with $x_\ell \in {\cal X}_\ell$, $\forall \ell \in \mathbb{N}_{0:L}$ satisfies the following condition: 
\begin{equation}
x_{\ell'} \in {\cal R}_j,\ \ell' \in \mathbb{N}_{1:L} \implies x_{\ell} \notin {\cal R}_i,\ \forall \ell \in \mathbb{N}_{\ell':L}, 
\end{equation}
which directly follows that (C.3) holds. Thus, if (D.1), (D.2) are satisfied, then the reachability holds from ${\cal R}_i$ to ${\cal R}_j$. This completes the proof. 
\end{proof}

\rthm{reachable_lem} indicates that reachability holds from ${\cal R}_i$ to ${\cal R}_j$, if the computed initial tube ${\cal X} _0$ is \textit{large enough} to satisfy ${\cal R}_i \subseteq {\cal X} _0$, and once some tube intersects ${\cal R}_j$, it does not intersect ${\cal R}_i$ afterwards. The latter condition is required to satisfy (C.3) in \rdef{reachable_def}, which aims to exclude the case when the state moves back to ${\cal R}_i$ after the entrance to ${\cal R}_j$. Consequently, reachability from ${\cal R}_i$ to ${\cal R}_j$ (i.e., if $(s_i, s_j) \in \delta$ holds) can be checked by solving \rpro{opt_problem} and checking if (D.1) and (D.2) hold. 

For the notational use in the next section, let $L_{x_0}$, $L_{x}$, $L_{u_0}$, $L_u : {\cal R}\times{\cal R} \rightarrow 2^{\cal X}$ be the mappings given by 
\begin{align}
L_{x_0} ({\cal R}_i, {\cal R}_j) &= {\cal X} _{0:L}, \ \ \ L_{x} ({\cal R}_i, {\cal R}_j) = {\cal X} _{1:L} \label{opttube}\\
L_{u_0} ({\cal R}_i, {\cal R}_j) &= {\cal U} _{0:L},\ \ \ L_{u} ({\cal R}_i, {\cal R}_j) = {\cal U} _{1:L}, \label{optinput}
\end{align}
i.e., the above mappings indicate the optimal sequence of tubes and the control sets as the solution to \rpro{opt_problem} for a given pair of two regions of interest. 

\begin{myrem}[On increasing possiblity to achieve reachability]
Since the initial tube is enlarged from \textit{only} a single point inside ${\cal R}_i$, it may be difficult to fulfill the condition ${\cal R}_i \subseteq {\cal X} _0$. 
To increase the chance to achieve reachability, one can generate nominal trajectories from \textit{multiple} points in ${\cal R}_i$, rather than only from a single point in ${\cal R}_i$. Namely, a set of state samples  $\hat{x}^{(1)} _0, \hat{x}^{(2)} _0, \ldots, \hat{x}^{(M)} _0$ are uniformly generated from ${\cal R}_i$, and for each $\hat{x}^{(m)} _0$, $m \in \mathbb{N}_{1:M}$, we produce the nominal trajectories and the tubes according to the above procedure. Then, if the union of all initial tubes contain ${\cal R}_i$, say, ${\cal R}_i \subseteq (\cup^M _{m=1}{\cal X}^{(m)} _0)$, toghther with the condition that (D.2) is satisfied, it is guarateed that reachability holds from ${\cal R}_i$ to ${\cal R}_j$. \qedwhite 
\end{myrem}

\subsection{Adding self-loops via invariance property}\label{invariance_sec}
In the previous subsection, we provide a way to analyze if $(s_i, s_j) \in \delta$ with $i\neq j$. To complete the characterization of $\delta$, one also needs to check if the self-loop exists, i.e., $(s_i, s_i) \in \delta$ for all $i\in\mathbb{N}_{1:N_I}$. 
Since the existence of the self-loop $(s_i, s_i) \in \delta$ implies that there exists a controller such that the state remains in ${\cal R}_i$ for all time, we need to check that the set ${\cal R}_i$ has an \textit{invariance property}, as the definition given below: 

\begin{mydef}[Robust controlled invariant set\cite{borrelli}]\label{lambda_contractive}
The set ${\cal R}_i$ is said to be a robust controlled invariant set, which we denote by $(s_i, s_i) \in \delta$, if and only if there exists a control law $\kappa : \mathbb{R}^n \rightarrow \mathbb{R}^m$ such that $x\in {\cal R}_i \implies Ax + B\kappa(x) + w \in {\cal R}_i$, $\forall w \in {\cal W}$. 
\end{mydef}

Since ${\cal R}_i$ is a polytope and we consider linear systems as in \req{sys}, the invariance property of ${\cal R}_i$ can be easily checked according to the following proposition (see, e.g., \cite{borrelli,blanchini1999a}): 
\begin{myprop}
Let $v_{1,i}, v_{2,i}, \ldots , v_{n_i, i}$ be all the vertices of ${\cal R}_i$. 
Then, ${\cal R}_i$ is a robust controlled invariant set if there exist $u_{1, i}, u_{2, i}, \ldots, u_{n_i, i} \in {\cal U}$ such that 
\begin{equation}\label{invariance_property}
A v_{n, i} + B u_{n, i} \in {\cal R}_i \ominus {\cal W},\ \forall n \in \mathbb{N}_{1:n_i}.
\end{equation} 
\end{myprop}
Therefore, the set ${\cal R}_i$ has an invariance property if we can find a set of control inputs such that each vertex of ${\cal R}_i$ can be driven into the interior of ${\cal R}_i$ that is tightened by ${\cal W}$. 

Suppose that we can find $u_{1, i}, u_{2, i}, \ldots, u_{n_i, i} \in {\cal U}$ such that \req{invariance_property} holds and let ${\cal U}_i = {\rm co} \{ u_{1, i}, u_{2, i}, \ldots, u_{n_i, i}\}$. As with \req{opttube} and \req{optinput}, we provide the mappings $L_{x_0}$, $L_{x}$, $L_{u_0}$, $L_u$ for the pair of the \textit{same} regions of interest be given by 
\begin{align}
L_{x_0} ({\cal R}_i, {\cal R}_i) &= {\cal R}_i, \ \ \ L_{x} ({\cal R}_i, {\cal R}_i) = {\cal R}_i \label{opttubesame}\\
L_{u_0} ({\cal R}_i, {\cal R}_i) &= {\cal U} _{i},\ \ \ L_{u} ({\cal R}_i, {\cal R}_i) = {\cal U} _{i}, \label{optinputsame}
\end{align}

\subsection{Summary of the algorithm}
From \rsec{reachability_sec}, it is shown that by executing the trajectory generation and the tube generation, we can characterize the transition relation $\delta$ for all $(s_i, s_j) \in S \times S$ with $i\neq j$. Moreover, from \rsec{invariance_sec} we can check if $(s_i, s_i ) \in \delta$ for all $i\in\mathbb{N}_{1:N_I}$ by evaluating if the set ${\cal R}_i$ has an invariance property. Thus, we can characterize $\delta$ for all $(s_i, s_j) \in S \times S$, and, as a consequence, we can construct the transition system ${\cal T}$ as an abstraction of the control system in \req{sys}. 

\section{Implementation} \label{implementation}
Based on the transition system ${\cal T}$ given in the previous section, we now present an online,  controller/communication synthesis algorithm as a solution to \rpro{problem_formulation}. 
Following the hierarchical-based approach, the proposed algorithm consists of \textit{high} and \textit{low level planning}. 
In particular, we present control and communication strategies in the low level planner, by utilizing the generated tubes obtained in the previous section. 

\subsection{High level planning}
In the high level planning, the controller produces an infinite sequence of the regions of interest that the state should follow to satisfy the formula $\phi$. Since the reachability among the regions of interest can be captured by the transition system ${\cal T}$, this can be done by finding a path $s_{seq} \in (2^{S})^\omega$ of ${\cal T}$, such that ${\rm trace}(s_{seq}) \models \phi$ holds. Although there exist several methodologies to achieve this, this paper adopts an {automata-based model checking} algorithm; in this paper, we only describe the overview of the approach and refer the reader to Chapter~5 in \cite{baier} for a more detailed explanation. First, we construct a B{\"u}chi automaton associated with the formula $\phi$ by using the translation tools, such as LTL2BA \cite{ltl2ba}. Since the automaton can be viewed as a graph consisting nodes and edges, we can combine it with ${\cal T}$ and find an accepting path by implementing graph search methodologies. In particular, we construct a product automaton between the B{\"u}chi automaton and ${\cal T}$, and find the strongly connected components through the depth-first search method. Moreover, if mutiple accepting paths are found, we select \textit{shortest} one that provides the smallest number of symbols to express the prefix-suffix structure (for details, see e.g.,  \cite{belta2008a}). By doing this, it holds that once the two consecutive symbolic states are the same, it remains the same afterwards. That is, denoting by $s^* _{seq} = s^* _0 s^* _1 s^*_2 \cdots $ the (shortest) accepting path, it holds that 
\begin{equation}\label{belta_result}
s^* _i = s^* _{i+1},{\rm for\ some}\ i\in\mathbb{N} \implies s^* _j = s^* _i,\ \forall j \geq i
\end{equation}
(see Proposition~2 in \cite{belta2008a} as the equivalent result). As we will see later, this property is utilized to show that the trace of the accepting path ${\rm trace}(s^* _{seq})$ is consistent with the one generated by the state trajectory. 

Now, once the high level planner finds the accepting path $s^* _{seq} = s^* _0 s^* _1 \cdots$ as above, we can obtain the corresponding sequence of the regions of interest that is projected from $s^* _{seq}$, i.e., 
\begin{equation}\label{region_sequence}
{\cal R}^* _{seq} = {\cal R}^* _0{\cal R}^* _1{\cal R}^* _2 \cdots, 
\end{equation}
where ${\cal R}^* _i = \Gamma (s^* _i)$, $\forall i \in \mathbb{N}$. Note that since $s^* _0 = s_{init}$, we have ${\cal R}^* _{0} = {\cal R}_{init}$. As described above, ${\cal R}^* _{seq}$ represents the infinite sequence of the regions of interest that the state trajectory should traverse to satisfy $\phi$. More specific control and communication design to achieve this is given in the low-level planner, as described in the next subsection. 


\subsection{Low level planning}
In the low level planning, the controller designs a sequence of control inputs as well as the communication times when the plant transmits the state information to the controller. 
To this end, let ${\cal X}^* _0 {\cal X}^* _1 {\cal X}^* _2 \cdots$, ${\cal U}^* _0 {\cal U}^* _1 {\cal U}^* _2 \cdots$ denote the infinite sequence of the tubes and the control sets produced by
\begin{equation}\label{optimal_polytopes}
\underbrace{L_{{x}_0} ({\cal R}^* _0, {\cal R}^* _1)}_{{\cal X}^* _0\ \cdots \ {\cal X}^* _{L_1} }\  \underbrace{L_{x} ({\cal R}^* _1, {\cal R}^* _2)}_{{\cal X}^* _{L_1+1}\ \cdots \ {\cal X}^* _{L_2}}\ \underbrace{L_x ({\cal R}^* _2, {\cal R}^* _3)}_{{\cal X}^* _{L_2+1}\ \cdots \ {\cal X}^* _{L_3}} \cdots, 
\end{equation}
\begin{equation}\label{optimal_inputs}
\underbrace{L_{{u}_0} ({\cal R}^* _0, {\cal R}^* _1)}_{{\cal U}^* _0\ \cdots \ {\cal U}^* _{L_1} }\  \underbrace{L_{u} ({\cal R}^* _1, {\cal R}^* _2)}_{{\cal U}^* _{L_1+1}\ \cdots \ {\cal U}^* _{L_2}}\ \underbrace{L_u ({\cal R}^* _2, {\cal R}^* _3)}_{{\cal U}^* _{L_2+1}\ \cdots \ {\cal U}^* _{L_3}} \cdots, 
\end{equation}
where for the notational simplicity we denote 
$L_{x_0} ({\cal R}^* _0, {\cal R}^* _1) = {\cal X}^* _0\ {\cal X}^* _2 \cdots {\cal X}^* _{L_{1}}$, 
$L_{u_0} ({\cal R}^* _0, {\cal R}^* _1) = {\cal U}^* _0\ {\cal U}^* _2 \cdots {\cal U}^* _{L_{1}}$, 
and 
\begin{align}
L_{x}({\cal R}^* _i, {\cal R}^* _{i+1})  = {\cal X}^* _{L_i+1} \cdots {\cal X}^* _{L_{i+1}},\ \ \forall i \in \mathbb{N}_+, \\
L_{u}({\cal R}^* _i, {\cal R}^* _{i+1})  = {\cal U}^* _{L_i+1} \cdots {\cal U}^* _{L_{i+1}},\ \ \forall i \in \mathbb{N}_+, 
\end{align}
with indices $L_i\in \mathbb{N}$, $i\in \mathbb{N}_+$ appropriately chosen to line up the sequences: 
\begin{align}
{\cal X}^* _0 {\cal X}^* _1 {\cal X}^* _2 {\cal X}^* _3 \cdots, \ \
{\cal U}^* _0 {\cal U}^* _1 {\cal U}^* _2 {\cal U}^* _3\cdots . 
\end{align}
Note that if ${\cal R}^* _i \neq  {\cal R}^* _{i+1}$, the mappings $\{L_{x_0}, L_{x}, L_{u_0}, L_u\} ({\cal R}^* _i, {\cal R}^* _{i+1})$ are given according to \req{opttube}, \req{optinput}, i.e., the tubes ${\cal X}^* _{L_i+1} \cdots {\cal X}^* _{L_{i+1}}$ are obtained by solving \rpro{opt_problem}. If ${\cal R}^* _i =  {\cal R}^* _{i+1}$ they are given according to \req{opttubesame}, \req{optinputsame}, i.e., ${\cal X}^* _{L_i+1} = {\cal R}^* _i$ and $L_{i+1} = L_i+1$. 
From above, ${\cal X}^* _0 {\cal X}^* _1 {\cal X}^* _2 \cdots$ indicates an infinite sequence of tubes that the state should follow to satisfy the formula $\phi$. 
In other words, if the trajectory starting from $x_0 
 \in {\cal R}_{init} (\subseteq {\cal X}^* _0)$ is forced to remain in the tubes, i.e., $x_k \in {\cal X}^* _k$, $\forall k\in\mathbb{N}$, the trajectory passes through all regions of interest ${\cal R}^* _0 {\cal R}^* _1{\cal R}^* _2 \cdots $, and thus the satisfaction of $\phi$ is achieved. As we will see below, the control sequence ${\cal U}^* _0 {\cal U}^* _1 {\cal U}^*_2 \cdots$ is utilized to design the suitable control inputs to achieve this. Since ${\cal X}^* _{k}$ and ${\cal U}^* _{k}$ for each $k\in\mathbb{N}$ are represented as a convex hull 
\begin{align}
{\cal X}^* _{k} = {\rm co} \{ x^* _{1, k}, \ldots, x^* _{p_k , k} \}, \ 
{\cal U}^* _{k} = {\rm co} \{ u^* _{1, k}, \ldots, u^* _{p_k , k} \} 
\end{align}
for all $k \in \mathbb{N}$, where $p_k \in \mathbb{N}_+$ represents the number of vertices of ${\cal X}^* _k$. Note that $p_k$ may be time-varying since the choice of ${\cal Z}_{ij}$ in \req{xsets} depends on the pair $({\cal R}_i, {\cal R}_j)$. 
The proposed control and communication strategies are recursively given according to the following algorithm: 

\smallskip
\noindent
\textbf{Algorithm~1} (Low level control/communication strategy): 
\begin{enumerate}
\item Initialization: Set $k=0$. 
\item The plant transmits $x_k$ to the controller. 
\item \textit{(Compute control inputs)}: For a given $H \in \mathbb{N}_+$, the controller computes $u_k, u_{k+1}, \ldots, u_{k+H-1} \in {\cal U}$ according to the following procedure: set $\hat{x}_k = x_k$ and for all $\ell \in \mathbb{N}_{0: H-1}$, 
\begin{align}
u_{k+\ell} &= \sum^{p_{k+\ell}} _{s=1} \lambda_{s, k+\ell} u^* _{s, k+\ell} \label{com_input} \\
\hat{x}_{k+\ell+1} &= A \hat{x}_{k+\ell} + B u_{k+\ell}, \label{com_state}
\end{align}
where $\lambda_{s, k+\ell} \in [0, 1]$, $s \in \mathbb{N}_{1:p_{k+\ell}}$ are such that $\hat{x}_{k+\ell} = \sum^{p_{k+\ell}} _{s=1} \lambda_{s, k+\ell}x^* _{s, k+\ell}$ with $\sum^{p_{k+\ell}} _{s=1} \lambda_{s, k+\ell} = 1$. 
\smallskip
\item \textit{(Compute the next communication time)}: The controller computes the next communication time $k+\ell^* _k$, where $\ell^* _k \in \mathbb{N}_{0:H}$ is determined as 
\begin{align}
 \ell^* _k = \max \{ \ell' \in \mathbb{N}_{0:H} |  \hat{x}_{k+\ell} \oplus {\cal W}_{k+\ell} \subseteq {\cal X}^* _{k+\ell}, \forall \ell \in \mathbb{N}_{0:\ell'} \}, \notag \\ \label{nextcom}
\end{align}
where ${\cal W}_{k+\ell} = \sum^{\ell} _{j=1} A^{j-1} {\cal W}$. 
\item The controller transmits $u_k, \ldots , u_{k+\ell^* _k -1}$ to the plant. The plant applies the obtained control sequence, and it sends the new state information $x_{k+\ell^* _k}$ to the controller. 
\item Set $k\leftarrow k+ \ell^* _k$ and go back to step (2). \qedwhite
\end{enumerate}

As shown in the above algorithm, for each communication time the controller computes a sequence of control actions until the given prediction horizon $H$, and the next communication time when the plant should again transmit the state information to the controller. 
The control sequence is firstly obtained by using vertex interporations of the generated tubes according to \req{com_input}, \req{com_state}. Applying such control strategy is motivated by following property: 
\begin{mylem}\label{lem_control_seq}
Suppose that $x_k \in {\cal X}^* _k$ and the control sequence $u_{k}, u_{k+1}, \ldots, u_{k+H-1}$ is given according to \req{com_input}, \req{com_state}. Then, we have $\hat{x}_{k+\ell} \in {\cal X}^* _{k+\ell} \ominus {\cal W}$, $\forall \ell \in \mathbb{N}_{1: H}$. 
\end{mylem}
\rlem{lem_control_seq} implies that by applying the control sequence according to \req{com_input}, all predictive states remain inside the corresponding tubes that are tightened by ${\cal W}$. For the proof, see Appendix. 


Once the controller computes the above control sequence, it proceeds to the computation of the next communication time $k+\ell^* _k$, where $\ell^* _k$ is given according to \req{nextcom}. As shown in the algorithm, the communication is given in a self-triggered manner as the controller iteratively computes the next communication time for each update time. The following lemma states that the next communication time is given such that the \textit{actual states} are guaranteed to remain inside the tubes: 
\begin{mylem}\label{lem_nextcom}
Suppose that $x_k \in {\cal X}^* _k$ and the control sequence $u_{k}, u_{k+1}, \ldots, u_{k + \ell^* _k-1}$ is applied at the plant for all $k, k+1, \ldots, k+\ell^* _k -1$. 
Then, for all $\ell \in \mathbb{N}_{1:\ell^* _k}$, the actual states in accordance with \req{sys} satisfies $x_{k+\ell} \in {\cal X}^* _{k+\ell}$, $\forall w_{k}, w_{k+1}, \ldots, w_{k+\ell-1} \in {\cal W}$. 
\end{mylem}
\begin{proof}
First, observe that the actual state at $k+\ell$, $\ell \in \mathbb{N}_{1: \ell^* _k}$ by applying $u_k, \ldots, u_{k+\ell-1}$ is given by 
\begin{align}
x_{k+\ell} &= A^{\ell} x_k + \sum^{\ell-1} _{j=0} A^{j} B u_{k+j} + \sum^{\ell-1} _{j=0} A^{j} w_{k+j} \notag \\
             & = \hat{x}_{k+\ell} + \sum^{\ell} _{j=1} A^{j-1} w_{k+\ell}. \label{actual_state}
\end{align}
That is, the set $\hat{x}_{k+\ell} \oplus {\cal W}_{k+\ell}$, ${\cal W}_{k+\ell} = \sum^{\ell} _{j=1} A^{j-1} {\cal W}$ indicates the set of \textit{all} actual states that can be reached at $k+\ell$ under the influence of disturbances. From \req{nextcom} it follows that $\hat{x}_{k+\ell} \oplus {\cal W}_{k+\ell} \subseteq {\cal X}^* _{k+\ell}$, $\forall \ell \in \mathbb{N}_{1:\ell^* _k}$, which means from \req{actual_state} that $x_{k+\ell} \in {\cal X}^* _{k+\ell}$, $\forall w_{k}, w_{k+1}, \ldots, w_{k+\ell -1} \in {\cal W}$, which holds for all $\ell \in \mathbb{N}_{1:\ell^* _k}$. Thus the proof is complete. 
\end{proof}

From the proof of \rlem{lem_nextcom}, we can ensure that $x_{k+\ell} \in {\cal X}^* _{k+\ell} $ for all $\ell \in \mathbb{N}_{1: \ell^* _k}$, meaning that applying the control sequence $u_k, \ldots, u_{k+\ell^* _k-1}$ in an open-loop fashion will not violate the formula $\phi$. 
For the remaining time steps $k+\ell$ with $\ell \in \mathbb{N}_{\ell^* _k + 1 : H}$, however, it is no longer guaranteed that we have $\hat{x} _{k+\ell} \oplus {\cal W}_{k+\ell} \subseteq {\cal X}^* _{k+\ell}$. This means that the actual states $x_{k+\ell}$ for $\ell \in \mathbb{N}_{\ell^* _k + 1 : H}$ are possible to leave the set ${\cal X}^* _{k+\ell}$, leading to the violation of $\phi$. Thus, $k+\ell^* _k$ represents the maximal future time step that the actual states are guaranteed to remain in the tubes to satisfy $\phi$, as well as that the next state measurement should be communicated to update the control inputs.

\begin{myrem}[On the positive inter-transmission times]\label{positive_transmission_rem}
From \rlem{lem_control_seq} and ${\cal W}_1 = {\cal W}$, it follows that $x_k \in {\cal X}_k \implies \hat{x}_{k+\ell} \oplus {\cal W}_1 \subseteq {\cal X}^* _{k+\ell}$. 
Thus, $\hat{x}_{k+\ell} \oplus {\cal W}_\ell \subseteq {\cal X}^* _{k+\ell}$ for all $\ell = 0, 1$, and, therefore, it holds that $\ell^* _k \geq 1$ (i.e., the inter-transmission time step is always positive). This property is important, since if we \textit{were} to obtain $\ell^* _k = 0$, the controller would not provide a suitable next communication time according to \req{nextcom}. Note that while $\ell^* _k \geq 1$ is guaranteed, we cannot guarantee $\ell^* _k \geq 2$. This is because we have ${\cal W}_\ell \supseteq {\cal W}$ for all $\ell \geq 2$ and thus we cannot ensure that the condition $\hat{x} (k+\ell) \oplus {\cal W}_\ell \subseteq {\cal X}^* (k+\ell)$ holds.  
\end{myrem}

\subsection{Summary of the implementation}
Since the low level planner appropriately computes control inputs and communication times so that the state trajectory follows the generated tube sequence, we obtain the following result: 
\begin{mythm}
Suppose that for given $x_0 \in {\cal R}_{init}$ and LTL formula $\phi$, the high level controller finds the sequence of the regions of interest as in \req{region_sequence}. Moreover, suppose that the low level planner (Algorithm~1) is implemented. 
Then, the resulting state trajectory ${\bf x} = x_0 x_1 x_2 \cdots $ satisfies $\phi$, i.e., ${\rm trace} ({\bf x}) \models \phi$. This holds for all $w_k \in {\cal W}$, $\forall k \in \mathbb{N}$. 
\end{mythm}

\begin{proof}
From \rlem{lem_nextcom}, it holds that $x_k \in {\cal X}^* _k \implies x_{k+\ell} \in {\cal X}^* _{k+\ell}, \forall \ell \in \mathbb{N}_{1:\ell^* _k}$. Thus, by induction, we obtain $x_k \in {\cal X}^* _k$, $\forall k\in\mathbb{N}$ with $x_0 \in {\cal R}_{init} \subseteq {\cal X}^* _0$. From \req{optimal_polytopes} the actual state traverses all the regions of interest ${\cal R}^* _{seq} = {\cal R}^* _0 {\cal R}^* _1 {\cal R}^* _2 \cdots $. To prove the theorem, we will show that ${\rm trace} ({\bf x}) = {\rm trace} (s^* _{seq})$. 
To this end, consider the following two cases: (i) it holds that ${\cal R}^* _i \neq {\cal R}^* _{i+1}$, $\forall i \in \mathbb{N}$;  (ii) there exists $i \in \mathbb{N}$ such that ${\cal R}^* _j \neq {\cal R}^* _{j+1}$, $\forall j \in \mathbb{N}_{0:i-1}$ and ${\cal R}^* _i = {\cal R}^* _{i+1}$. For the case (i), each tube sequence from ${\cal R}^* _i$ to ${\cal R}^* _{i+1}$ ($i\in \mathbb{N}$) has been obtained by solving \rpro{opt_problem}, and the resulting state trajectory satisfies (C.1)-(C.3) in \rdef{reachable_def}. In other words, the trace of the state trajectory while moving from ${\cal R}^* _i$ to ${\cal R}^* _{i+1}$ is $g(s^* _i) g(s^* _{i+1})$ for all $i \in \mathbb{N}$, which is consistent with the trace from $s^* _i$ to $s^* _{i+1}$ (see \rprop{trace_prop}). Thus, we obtain ${\rm trace} ({\bf x}) = {\rm trace} (s^* _{seq}) \models \phi$ and so the satisfaction of $\phi$ is guaranteed. Now, consider the case (ii). From above, the trace of the state trajectory until it reaches ${\cal R}^* _i$ ($x_k , k\leq L_i$) is $g(s^* _0) g(s^* _1) \cdots g(s^* _i)$, which is consistent with the trace of the path $s^* _0 s^* _1 \cdots s^* _i$. To illustrate the consistency after $i$, observe that we have ${\cal R}^* _j = {\cal R}^* _i$, $\forall j \geq i$ from \req{belta_result}. This implies that once the state enters ${\cal R}^* _i$ it remains there for all time, i.e., $x_k \in {\cal R}^* _i$, $\forall k \geq L_i + 1$. Thus, by employing the rule (3) in \rdef{trace_traj} the trace of the state trajectory for $x_k$, $k \geq L_i + 1$ is $g(s^* _i) g(s^* _{i+1}) \cdots $ with ${s}^* _j = {s}^* _i$, $\forall j \geq i$. As a consequence, the resulting state trajectory ${\bf x} = x_0 x_1 x_2 \cdots$ satisfies ${\rm trace} ({\bf x}) = {\rm trace} (s^* _{seq}) \models \phi$. Therefore, regardless of the case (i) and (ii), it holds that ${\rm trace} ({\bf x}) = {\rm trace} (s^* _{seq}) \models \phi$. The proof is complete. 
\end{proof}



\section{Illustrative example} 
Consider the following system (see \cite{belta2008a}): 
\begin{equation} \label{sim_system}
\dot{{x}}  =  \left [
\begin{array}{cc}
0.2  &  -0.3   \\
0.5 &  -0.5
\end{array}
\right ] x + \left [
\begin{array}{cc}
1 & 0 \\
0 & 1
\end{array}
\right ] u + w, 
\end{equation}
where $x \in \mathbb{R}^2$ is the state, $u \in \mathbb{R}^2$ is the control input, and $w \in \mathbb{R}^2$ is the disturbance. We discretize \req{sim_system} under a $0.05$ sample-and-hold controller to obtain the corresponding discrete-time system in \req{sys}. We assume $x_{0} = [-4; -4] \in {\cal R}_2$ and the control and disturbance size set in \req{uwsets} are given by ${\cal U} = \{ u_k \in  {\mathbb{R}} \ |\ ||u_k||_{\infty} \leq 6.0 \}$ and ${\cal W} = \{ w_k \in  {\mathbb{R}}^{2}\ |\ ||w_k|| \leq 0.15\}$. The set ${\cal X} \subset \mathbb{R}^2$ is shown in \rfig{state_space_result}. In the figure, the white regions represent the free space in which the state can move freely, and the black regions represent obstacles to be avoided. As shown in \rfig{state_space_result} there exist 4 regions of interest ${\cal R}_1, {\cal R}_2, {\cal R}_3, {\cal R}_4$, which are all $1\times 1$ squares. 


For example, consider reachability from ${\cal R}_2$ to ${\cal R}_3$. In the simulation, the nominal trajectory is first generated by implementing standard RRT algorithm \cite{lavalle1999}. The tubes are generated around the nominal trajectory by solving \rpro{opt_problem} and is illustrated in \rfig{state_tube_result2} as the sequence of blue regions. When solving \rpro{opt_problem}, we assume ${\cal Z} = \{z \in \mathbb{R}^2\ |\ ||z||_\infty \leq 1 \}$. The figure illustrates that the initial tube includes ${\cal R}_2$, as well as that once the tube intersects ${\cal R}_3$ it does not intersect ${\cal R}_2$ afterwards. Therefore, it is shown from \rthm{reachable_lem} that reachability holds from ${\cal R}_2$ to ${\cal R}_3$. The generated tubes from ${\cal R}_3$ to ${\cal R}_4$ and from ${\cal R}_4$ to ${\cal R}_1$ are also illustrated in \rfig{state_tube_result3} and \rfig{state_tube_result4}, respectively, and we can also show the reachability for these regions. 
Similarly, we analyze reachability for all the other pairs of the regions of interest and construct the transition system. The resulting transition system ${\cal T}$ contains 4 symbolic states and 16 transitions. The total time to construct ${\cal T}$ is 950s on Windows 10, Intel(R) Core(TM) 2.40GHz, 8GB RAM. 

\begin{figure}[t]
   \centering
    \subfigure[Free space ${\cal X}$]{
      {\includegraphics[width=4.2cm]{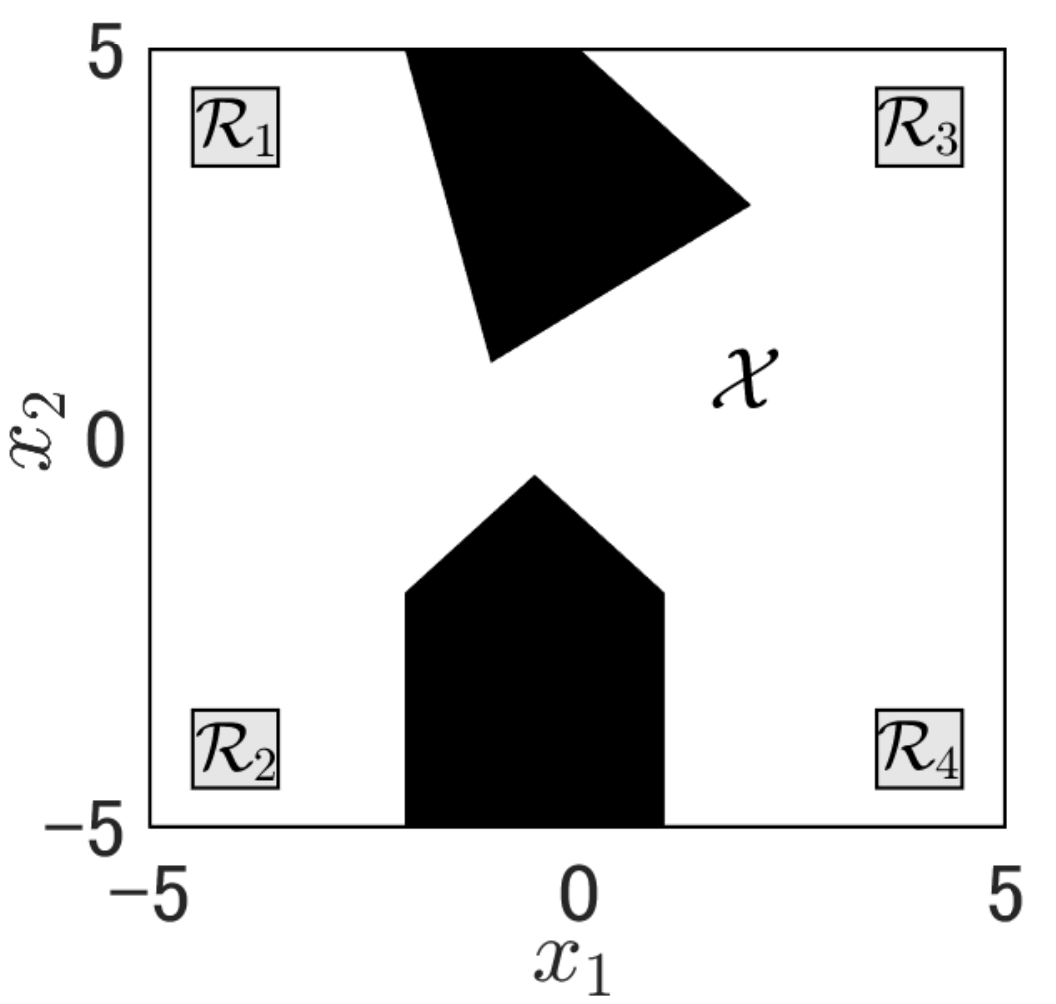}} \label{state_space_result}}
      \hspace{-0.58cm}
   \subfigure[Tubes from ${\cal R}_2$ to ${\cal R}_3$]{
      {\includegraphics[width=4.2cm]{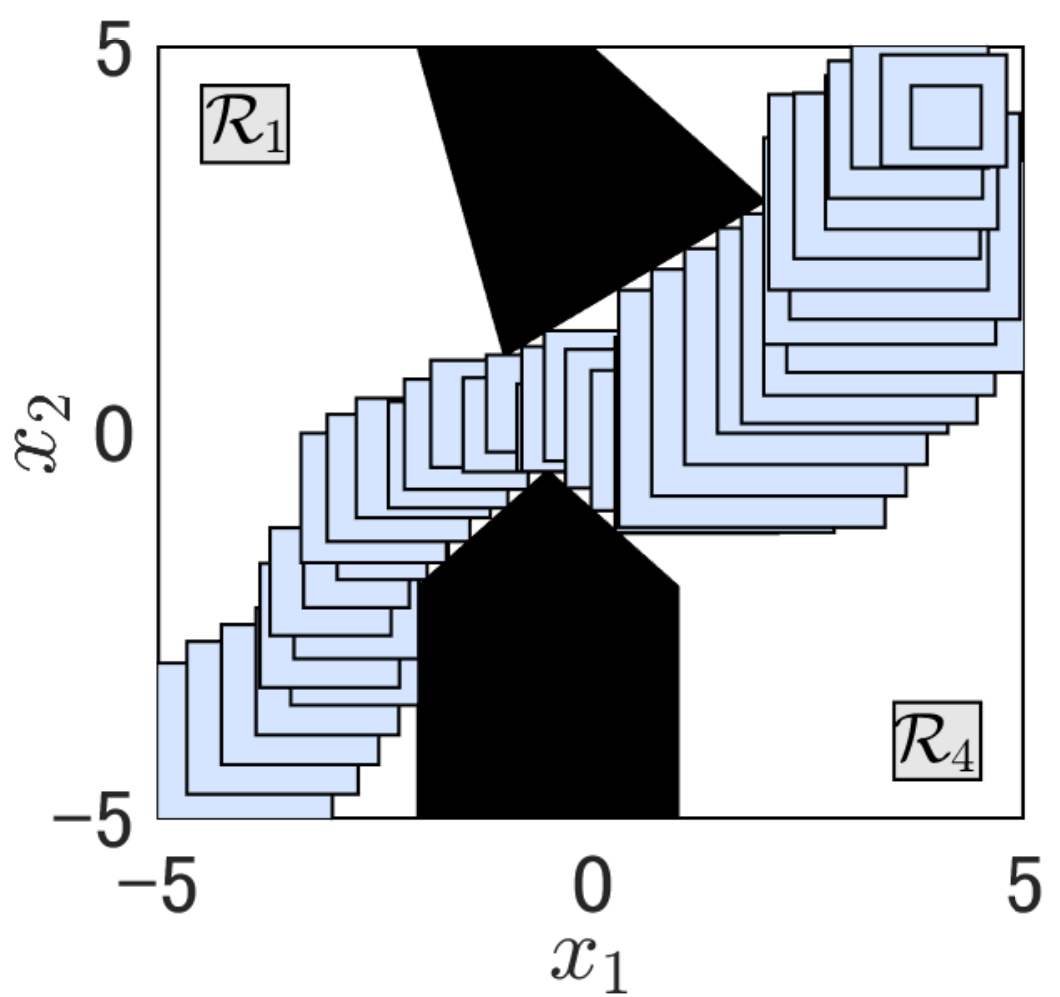}} \label{state_tube_result2}}
   \subfigure[Tubes from ${\cal R}_3$ to ${\cal R}_4$]{
      {\includegraphics[width=4.2cm]{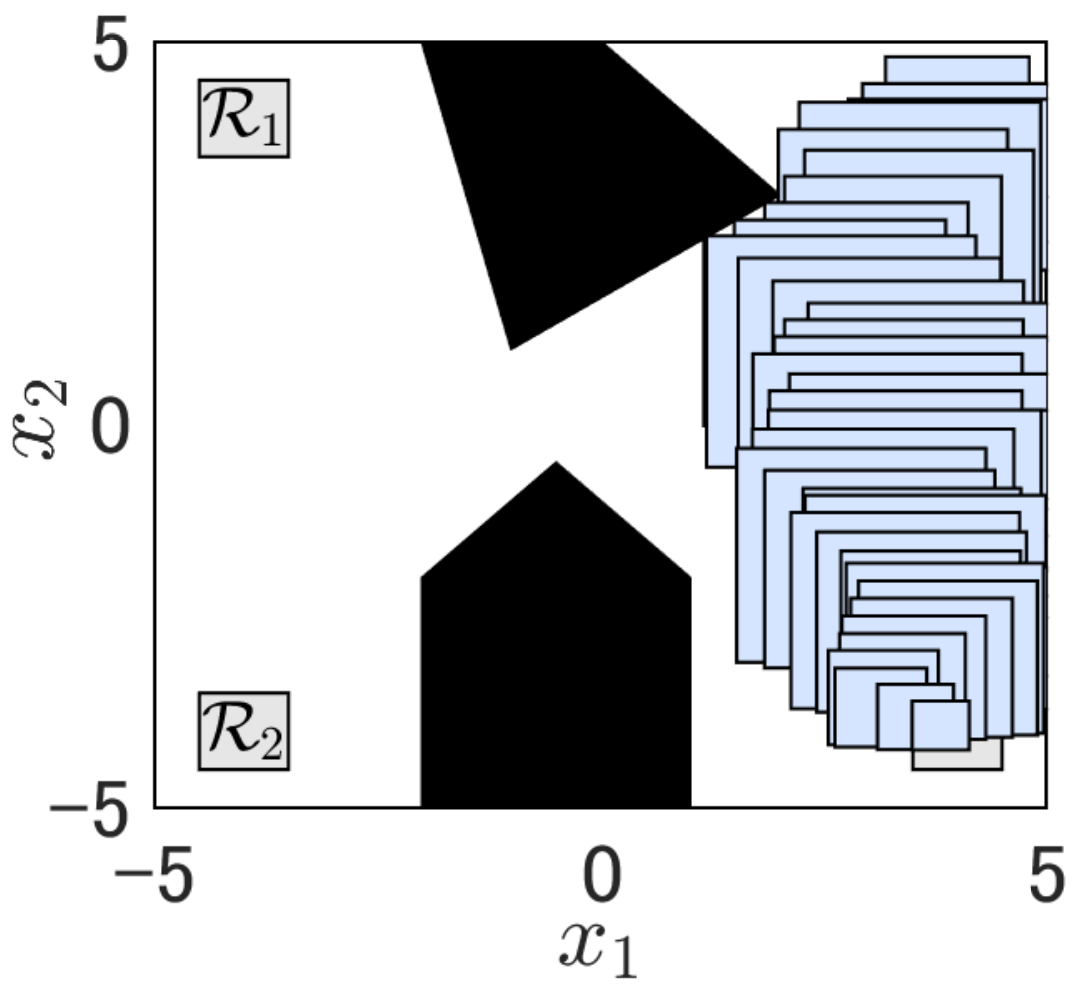}} \label{state_tube_result3}}
      \hspace{-0.55cm}
   \subfigure[Tubes from ${\cal R}_4$ to ${\cal R}_1$]{
      {\includegraphics[width=4.2cm]{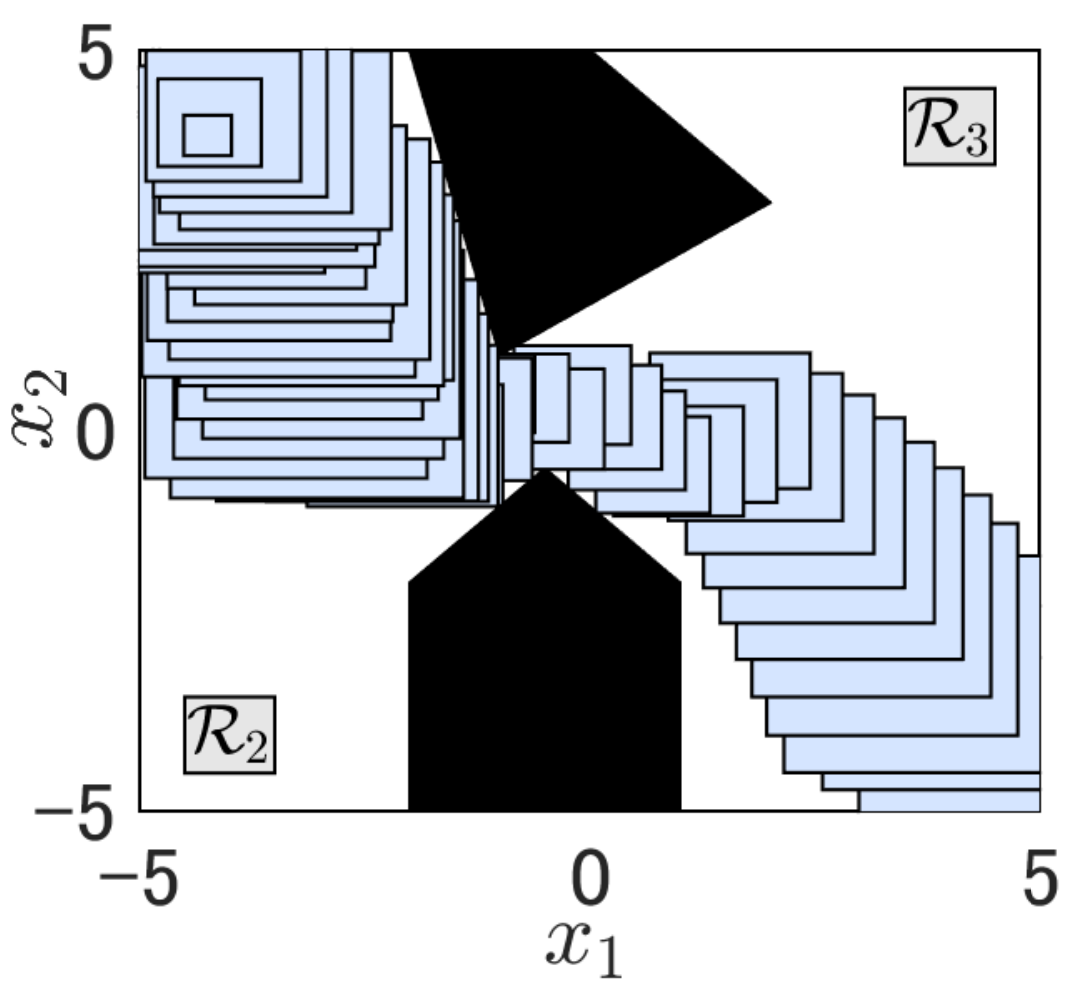}} \label{state_tube_result4}}
    \caption{Illustration of ${\cal X}$ and the generated tubes by solving \rpro{opt_problem}. }
 \end{figure}
 
 \begin{figure}[t]
   \centering
    \subfigure[$\phi = \phi_1$ (${k\in[0, 40]}$)]{
      {\includegraphics[width=4.3cm]{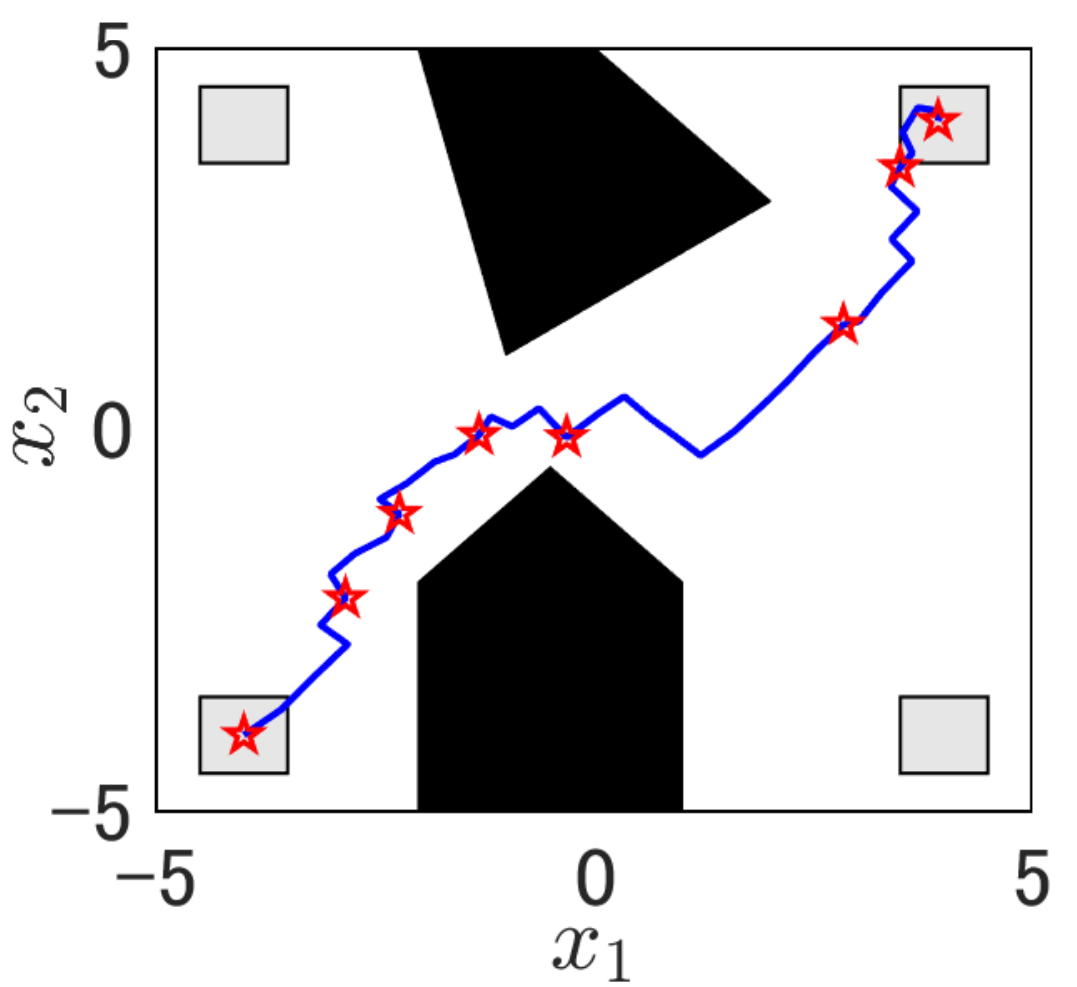}} \label{state_traj_result2}}
    \hspace{-0.5cm}
    \subfigure[$\phi = \phi_1$ (${k\in[0, 800]}$)]{
      {\includegraphics[width=4.3cm]{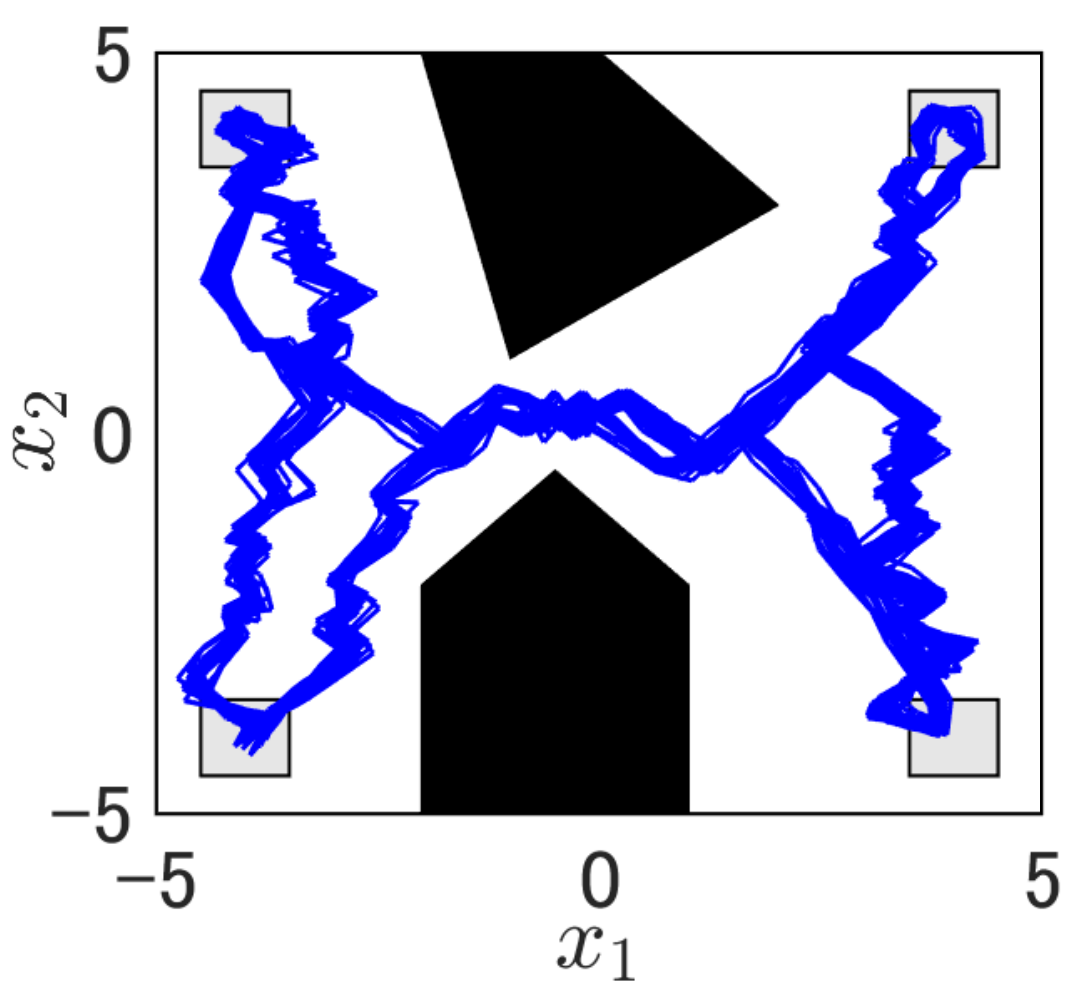}} \label{state_traj_result1}}
      \hspace{-0.5cm}
     \subfigure[$\phi = \phi_2$ (${k\in[0, 800]}$)]{
      {\includegraphics[width=4.2cm]{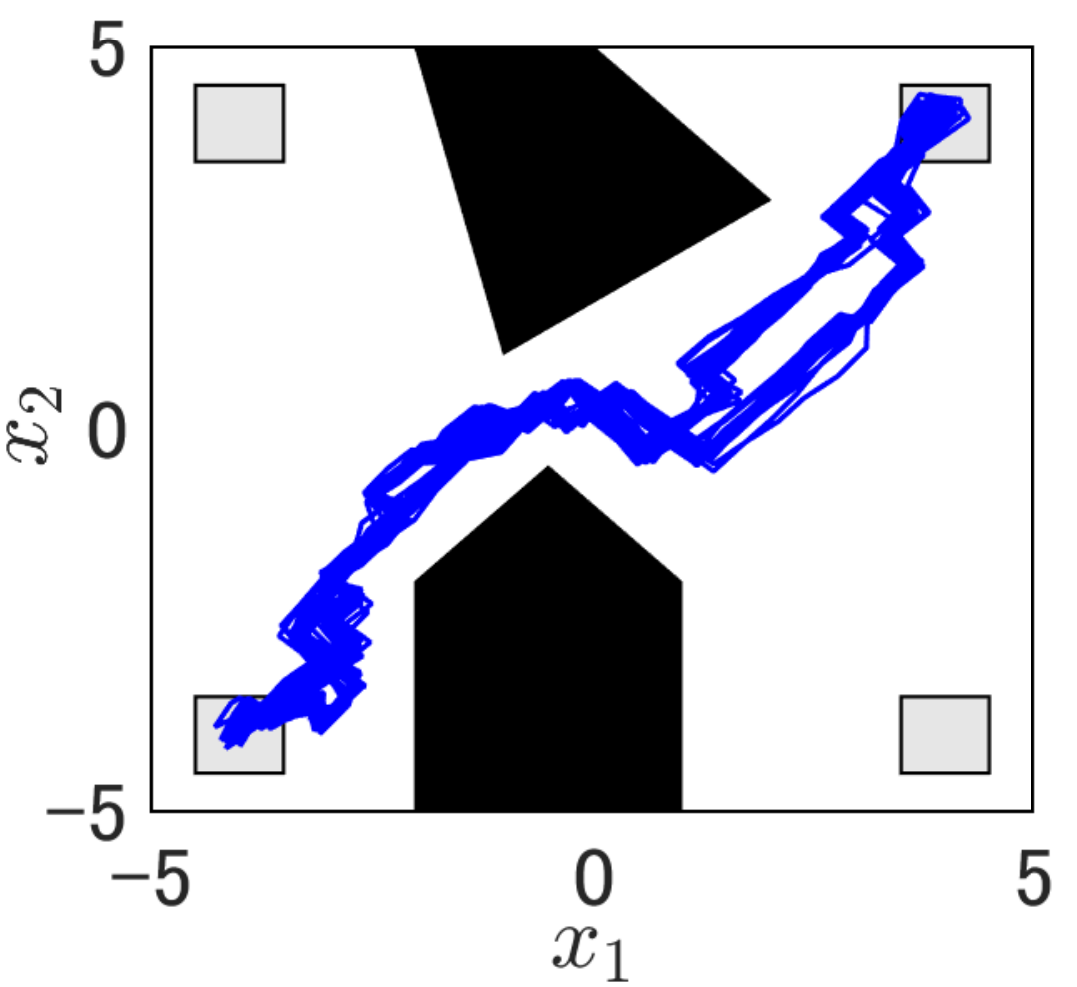}} \label{state_traj_result3}}
      \hspace{-0.5cm}
     \subfigure[$\phi = \phi_3$ (${k\in[0, 800]}$)]{
      {\includegraphics[width=4.3cm]{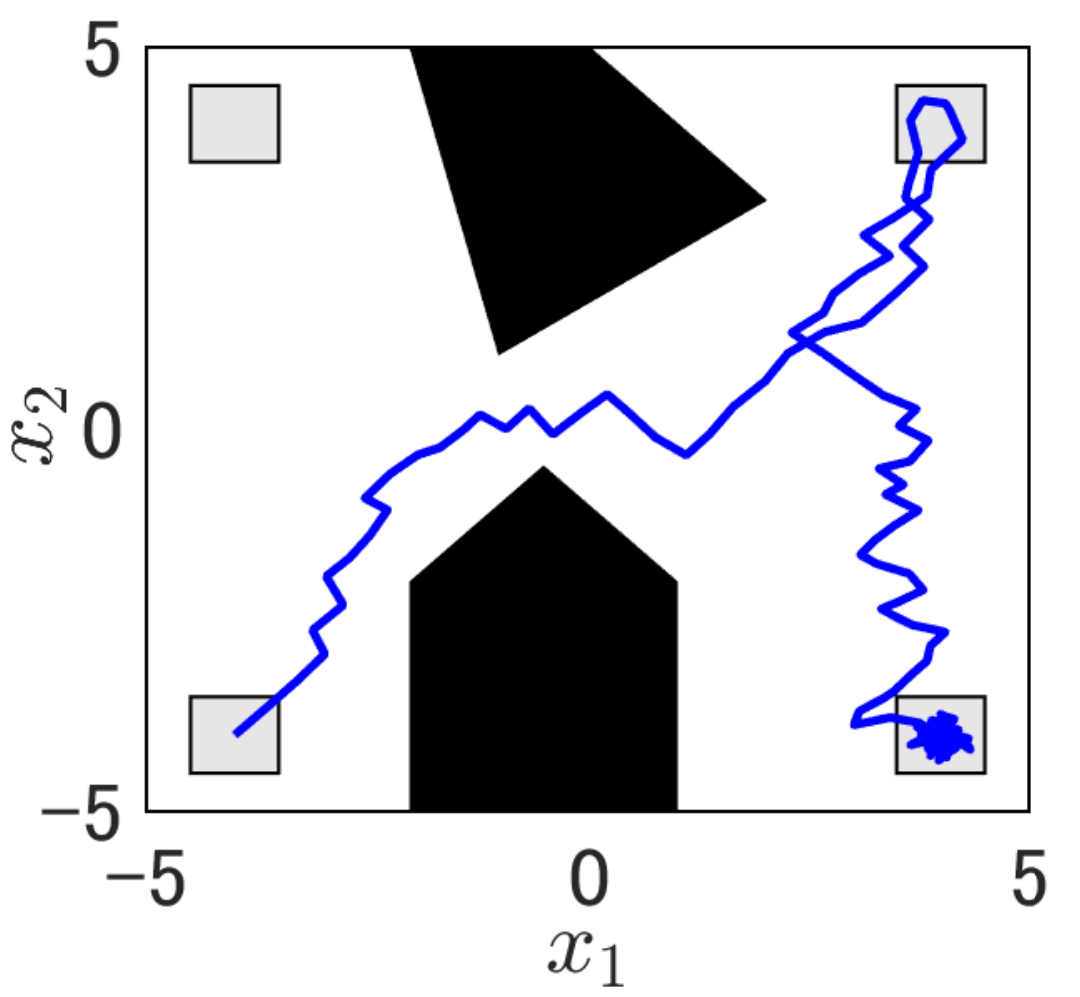}} \label{state_traj_result4}}
      \caption{State trajectories by applying Algorithm~1 with $\phi = \phi_1$ and $\phi = \phi_2$. }
    \label{state_traj_result}
\end{figure}


To illustrate the proposal, we consider the following three specifications: $\phi_1 = \Box ( \Diamond \pi_1 \wedge \Diamond \pi_2 \wedge \Diamond \pi_3 \wedge \Diamond \pi_4)$ (visit all regions infinitely often), $\phi_2 = \Box ( \Diamond \pi_2 \wedge \Diamond \pi_3)$ (visit ${\cal R}_2$ and ${\cal R}_3$ infinitely often), and $\phi_3 = \Diamond \pi_3 \wedge \Diamond \Box \pi_4$ (visit ${\cal R}_3$ and ${\cal R}_4$ and stay in ${\cal R}_4$ for all time). \rfig{state_traj_result} illustrates the resulting state trajectories (blue lines) by implementing Algorithm~1 with $\phi = \phi_1$ for $k\in[0, 40]$ (\rfig{state_traj_result2}) and $k\in[0, 800]$ (\rfig{state_traj_result1}). During the implementation, we set the prediction horizon as $H=10$. In \rfig{state_traj_result2}, red star marks represent the communication instants when control inputs are updated. The number of transmission time instants is $8$ during $k\in[0, 40]$, which indicates that the communication instants are 5 times smaller than the number of time steps. It can be seen from the figure that the transmission frequency becomes high when the state is close to the obstables and also to ${\cal R}_3$. This is possibly because the generated tubes tend to have small volumes (\rfig{state_tube_result2}) when the nominal trajectory is close to the obstacles and ${\cal R}_3$, and so the corresponding triggering conditions in \req{nextcom} become less likely to be satisfied. From \rfig{state_traj_result1}, it is shown that states are appropriately steered to satisfy the formula $\phi_1$ by surveiling all regions of interest. 
Similarly, \rfig{state_traj_result3} and \rfig{state_traj_result4} illustrate the state trajectories by applying Algorithm~1 with $\phi = \phi_2$ and $\phi = \phi_3$, respectively. From the figure, it is also shown that states are appropriately steered to satisfy the two formulas. 

The number of transmission instants during $k \in [0, 800]$ is given by $121$ ($\phi = \phi_1$), $141$ ($\phi = \phi_2$) and $190$ ($\phi = \phi_3$), while the periodic communication ($\ell^* _k = 1 , \forall k\in\mathbb{N}$) is $801$. In summary, it achieves less communication load than the periodic scheme, and the effectiveness of the proposed approach is validated. 


\section{Conclusion}
In this paper, we propose a self-triggered strategy for achieving both communication reduction for NCSs and the satisfaction of the temporal logic formulas. 
We start by constructing a finite transition system based on tube-based strategy. The generated tubes are utilized as a guidance to achieve the satisfaction of the temporal logic formula, as well as to design the communication strategy to achieve the communication reduction for NCSs. Finally, we illustrate the benefits of the proposed approach through a numerical simulation.


\balance

\appendix 
\noindent
\textbf{(Proof of \rlem{lem_control_seq})}:  
Suppose that $\hat{x} _{k+\ell} \in {\cal X}^* _{k+\ell}$, $\ell \in \mathbb{N}_{1:H-1}$ and ${\cal X}^* _{k+\ell} \in L_{x} ({\cal R}^* _i, {\cal R}^* _{i+1})$ (or $L_{x_0} ({\cal R}^* _i, {\cal R}^* _{i+1}))$ for some $i$. This means that ${\cal X}^* _{k+\ell}$ is one of the tubes generated by analyzing reachability from ${\cal R}^*  _i$ to ${\cal R}^* _{i+1}$. For the next tube ${\cal X}^* _{k+\ell+1}$, either the following case can be considered: 
(i) ${\cal X}^* _{k+\ell+1} \in L_{x} ({\cal R}^* _i, {\cal R}^* _{i+1})$; 
(ii) ${\cal X}^* _{k+\ell+1} \in L_{x} ({\cal R}^* _{i+1}, {\cal R}^* _{i+2})$. For the case (i), ${\cal X}^* _{k+\ell+1}$ is generated by analyzing the reachability for the same pairs of the regions of interest as ${\cal X}^* _{k+\ell}$. 
Since $L_{x} ({\cal R}^* _{i}, {\cal R}^* _{i+1})$ is a set of tubes obtained by solving \rpro{opt_problem}, it follows that 
\begin{align}
\hat{x}_{k+\ell+1} &= A \hat{x}_{k+\ell} + B u_{k+\ell} \notag \\
                       &= \sum^{p_{k+\ell}} _{s=1} \lambda_{s,k+\ell} ( A x^* _{s,\ell} + B u^* _{s,\ell} ) \in  {\cal X}^* _{k+\ell+1} \ominus {\cal W}, \label{pred_cond}
\end{align}
where $u_{k+\ell}$ and $\lambda_{s,k+\ell}$, $s \in \mathbb{N}_{1:p}$ are computed as in \req{com_input}. 
The last inclusion follows from the fact that $Ax^* _{s,\ell} + B u^* _{s, \ell} \in {\cal X} _{\ell+1} \ominus {\cal W}$, $\forall s \in \mathbb{N}_{1:p}$ from the constraint in \req{constraint}. Case (ii) implies that the tube ${\cal X}^* _{k+\ell}$ is the last element of $L_{x} ({\cal R}^* _i, {\cal R}^* _{i+1})$ (i.e., ${\cal X}^* _{k+\ell}\subseteq {\cal R}^* _{i+1}$) and ${\cal X}^* _{k+\ell+1}$ is the first element of $L_{x} ({\cal R}^* _{i+1}, {\cal R}^* _{i+2})$. 
Since $\hat{x}^* _{k+\ell} \in {\cal X}^* _{k+\ell}\subseteq {\cal R}^* _{i+1}$ and $L_{x} ({\cal R}^* _{i+1}, {\cal R}^* _{i+2})$ is a set of tubes obtained by solving \rpro{opt_problem}, we obtain \req{pred_cond}. 
Therefore, regardless of the case (i), (ii), we obtain 
$\hat{x}_{k+\ell} \in {\cal X}^* _{k+\ell} \implies \hat{x}_{k+\ell+1} \in {\cal X}^* _{k+\ell+1} \ominus {\cal W}$. By induction, this follows that $x_k (=\hat{x}_k) \in {\cal X}^* _{k} \implies \hat{x}_{k+\ell} \in {\cal X}^* _{k+\ell} \ominus {\cal W}$, $\forall \ell \in \mathbb{N}_{1:H}$, which completes te proof. 

\end{document}